\documentclass[a4paper,12pt]{article}

\usepackage[utf8]{inputenc}
\usepackage{amsfonts}
\usepackage{amssymb}
\usepackage{amsthm}
\usepackage{amsmath}
\usepackage{multicol}
\usepackage{enumerate} 
\usepackage{enumitem} 
\usepackage{caption}

\usepackage[a4paper,twoside,top=2.5cm, bottom=2cm, left=2.3cm, right=2.3cm]{geometry} 

\usepackage{graphicx}
\usepackage{dsfont}
\usepackage{authblk}

\usepackage{color}
\definecolor{hanblue}{rgb}{0.27, 0.42, 0.81}
\definecolor{red}{rgb}{1.0, 0.0, 0.0}
\usepackage[colorlinks, citecolor=blue,linkcolor=blue, urlcolor = blue]{hyperref}

\usepackage[english,capitalize]{cleveref}

\allowdisplaybreaks

\newcommand*{\uthe}{\mathbf{u}(\theta)}
\newcommand*{\duthe}{\mathbf{u}'(\theta)}
\newcommand*{\constant}{c}

\newcommand*{\de}{\,\ensuremath{\mathrm d}}

\theoremstyle{plain}

\newtheorem{thm}{Theorem}[section]
\newtheorem{lem}[thm]{Lemma}
\newtheorem{prop}[thm]{Proposition}
\newtheorem{cor}[thm]{Corollary}

\theoremstyle{definition}
\newtheorem{defn}[thm]{Definition}

\newtheorem*{teocit}{Theorem}
\newtheorem*{lemcit}{Lemma}

\theoremstyle{remark}
\newtheorem{rem}[thm]{Remark}

\numberwithin{equation}{section}

\begin{document}
\title{The Fourier transform of planar convex bodies\\and discrepancy over intervals of rotations}
\author{Thomas Beretti}
\affil{International School for Advanced Studies (SISSA)}
\date{\today}
\maketitle

\begin{abstract}
\noindent 
This work studies the Fourier transform of the characteristic function of planar convex bodies averaged over affine transformations. We establish lower and upper bounds on the latter quantities in terms of the geometric properties of the bodies considered. The second matter of study is the affine quadratic discrepancy of planar convex bodies, and we present sharp results on its asymptotic behaviour. In particular, we address averages over intervals of rotations, answering an open question of Bilyk and Mastrianni.
\end{abstract}

\tableofcontents

\section{Introduction}

The theory of irregularities of distribution, also known as discrepancy theory, concerns the approximation of the Lebesgue measure through samplings by Dirac deltas. One can equivalently pose such a question in an Euclidean space or a periodic setting. We introduce some basic notations for the latter. For a real positive number $p$, we define the one-dimensional torus with period $p$ and the unitary two-dimensional torus, respectively, as
\begin{equation*}
\mathbb{T}_p=\mathbb{R}/p\mathbb{Z}\quad\text{and}\quad\mathbb{T}^2=\mathbb{R}^2/\mathbb{Z}^2.
\end{equation*}
Moreover, we define the \emph{ordered-distance function}
        \begin{equation}\label{D1}
        \eta_p\colon\mathbb{T}_p\times\mathbb{T}_p\to[0,p)
        \end{equation}
        in such a way that
	 \begin{equation*}
	 \eta_p(x_1,x_2)=y\quad\text{if and only if}\quad x_1+y\equiv x_2\!\!\!\pmod p.
	 \end{equation*}
\par
 We introduce convenient notation on limit behaviours. Consider an unbounded set $U\subseteq[0,+\infty)$ and let $f$ and $g$ be two positive functions defined on $U$, then we say that it holds
 \begin{equation}\label{rel} 
     f(x)\preccurlyeq g(x)
 \end{equation}
 to intend that there exists a positive value $\constant$ such that
 \begin{equation*}
     \limsup_{x\to+\infty}\frac{f(x)}{g(x)}\leq \constant.
 \end{equation*}
 Moreover, in the case of $f_y$ and $g_y$ depend on a variable $y\in V\subseteq \mathbb{R}$, then we say that \eqref{rel} holds uniformly for every $y\in V$ to intend that the involved value $c$ does not depend on $y$.
 Last, if \eqref{rel} holds in both senses, then we say that it holds
 \begin{equation*}
     f(x)\asymp g(x).
 \end{equation*}
 \par
 We introduce suitable notation on affine transformations of the Euclidean plane. For an angle $\theta\in\mathbb{T}_{2\pi}$, we let $\sigma_\theta\colon\mathbb{R}^2\to\mathbb{R}^2$ be the counterclockwise rotation by $\theta$, and we set 
    \begin{equation*}
        \uthe=(\cos\theta, \sin\theta)
    \end{equation*}
    to be the unit vector in $\mathbb{R}^2$ that makes an angle $\theta$ with the $x$-axis. We let $\boldsymbol{\tau}\in\mathbb{R}^2$ be a translation factor and $\delta\geq0$ be a dilation factor. For a bounded set $\Omega\subset\mathbb{R}^2$, we define the action of an affine transformations on $\Omega$ by
    \begin{equation*}
    [\boldsymbol{\tau},\delta,\theta]\Omega=\boldsymbol{\tau}+\delta\sigma_{\theta}\Omega,
    \end{equation*}
with the convention that if a transformation is null, we omit its writing in the square brackets. Moreover, we let $\mathds{1}_{\Omega}$ stand for the characteristic (indicator) function of $\Omega$. We define the Fourier transform of $\mathds{1}_{\Omega}$ as
\begin{equation*}
    \widehat{\mathds{1}}_\Omega(\boldsymbol{\xi})=\int_{\Omega}e^{-2\pi i \mathbf{x}\cdot\boldsymbol{\xi}}\de \mathbf{x},
\end{equation*}
and from classic properties of the Fourier transform, we get that
\begin{equation}\label{FourierProp}
    \widehat{\mathds{1}}_{[\delta,\theta]\Omega}(\boldsymbol{\xi})=\delta^2\widehat{\mathds{1}}_\Omega(\delta\sigma_{-\theta}\boldsymbol{\xi}).
\end{equation}
\par 
We introduce a tool that allows us to switch from an Euclidean setting to a periodic one. Consider the periodization functional ${\mathfrak{P}}\colon L^1(\mathbb{R}^2)\to L^1(\mathbb{T}^2)$ defined in the sense that
\begin{equation*}
    {\mathfrak{P}}\{\mathds{1}_\Omega\}(\mathbf{x})=\sum_{\mathbf{n}\in\mathbb{Z}^2}\mathds{1}_\Omega(\mathbf{x}+\mathbf{n}).
\end{equation*}
We give the following notions of discrepancy.
\begin{defn}
Let $\Omega\subset\mathbb{R}^2$ and let $\mathcal{P}_N\subset\mathbb{T}^2$ be a set of $N$ points. We define the \emph{discrepancy} of $\mathcal{P}_N$ with respect to $\Omega$ as 
    \begin{equation}\label{Discrepancy1}
    \mathcal{D}(\mathcal{P}_N,\, \Omega)=\sum_{\mathbf{p}\in\mathcal{P}_N}{\mathfrak{P}}\{\mathds{1}_\Omega\}(\mathbf{p})-N|\Omega|.
\end{equation}
Further, let $I\subseteq\mathbb{T}_{2\pi}$ be an interval of angles. We define the \emph{affine quadratic discrepancy} of $\mathcal{P}_N$ with respect to $\Omega$ and $I$ as
\begin{equation}\label{Discrepancy2}
    \mathcal{D}_2(\mathcal{P}_N,\, \Omega,\,I)=\int_{I}\int_{0}^{1}\int_{\mathbb{T}^2}\left|\mathcal{D}( \mathcal{P}_N,\,[\boldsymbol{\tau},\delta,\theta]\Omega)\right|^2d\boldsymbol{\tau}\de \delta\de \theta.
\end{equation}
In particular, we will always assume that the interval of angles $I$ is such that $|I|>0$ (that is, $I$ is non-trivial).
\end{defn}

This paper aims to explore the affine quadratic discrepancy of planar convex bodies, namely, bounded convex sets of $\mathbb{R}^2$ with a non-empty interior. This question is deeply related to the asymptotic behaviour of the Fourier transform of the bodies considered. In turn, obtaining optimal estimates on the latter recovers meaningful geometric quantities. Before addressing our results on the Fourier transform, which the reader may find of independent interest, we describe the ones on the discrepancy.

\section{Main results: Discrepancy over Affine Transformations}

In Section~\ref{S3}, we prove our main results on the affine quadratic discrepancy. It turns out that the best estimates depend solely on the measure of the interval of angles considered and on a geometric quantity of the body. First, we give auxiliary definitions concerning points at the boundary.
    \begin{defn}\label{D2}
         Let $C\subset\mathbb{R}^2$ be a convex body. We set
    \begin{equation*}
        \boldsymbol{\Gamma}_C\colon \mathbb{T}_{\left|\partial C\right|}\to\mathbb{R}^2
    \end{equation*}
    to be the arc-length parameterization of the boundary $\partial C$. For $s\in\mathbb{T}_{\left|\partial C\right|}$, we define the \emph{set of normals} at $s$ as
    \begin{equation*}
    \nu_C(s)=\left[\nu_C^-(s),\nu_C^+(s)\right]=\left\{\theta\in\mathbb{T}_{2\pi}\,\colon\,\min_{\mathbf{a}\in C}\left(\mathbf{a}\cdot\uthe\right)=\boldsymbol{\Gamma}(s)\cdot\uthe\right\},
    \end{equation*}
    with the convention that if $\nu_C(s)$ is a single angle, then we simply consider
    \begin{equation*}
        \nu_C^-(s)=\nu_C^+(s)=\nu_C(s).
    \end{equation*}
    In particular, we say that $s\in\mathbb{T}_{\left|\partial C\right|}$ is an angular point if $\nu_C^-(s)\neq\nu_C^+(s)$.
    \end{defn}

    \begin{rem}
        Notice that if $C$ has a $\mathcal{C}^1$ boundary, then it has no angular point. Moreover, if $\partial C$ has strictly positive curvature, then $\nu_C$ is a bijection between $\mathbb{T}_{\left|\partial C\right|}$ and $\mathbb{T}_{2\pi}$. Last, if we choose $C$ to be an axis-symmetric square, then its vertices identify angular points whose sets of normals are $[(n-1)\pi/2, n\pi/2]$, with $n$ being an integer such that $1\leq n\leq4$. 
    \end{rem}
    
    For a generic set $A$, we write ${\rm int}(A)$ to denote its interior. Hence, we describe the aforementioned geometric quantity.
\begin{defn}\label{D3}
    Let $C\subset\mathbb{R}^2$ be a convex body. We define the \emph{angular trace} of $C$ as
    \begin{equation*}
    \mathcal{T}_C=\bigcup_{s\in\mathbb{T}_{|\partial C|}}{\rm int}\left(\nu_C(s)\right),
    \end{equation*}
    and further, we define the \emph{symmetric angular threshold} of $C$ as
    \begin{equation*}
    \psi_C=\max\left\{|J|\,\colon\, {J \text{ is a connected component of }\mathcal{T}_C}\cap\left(\mathcal{T}_C+\pi\right)\right\}.
    \end{equation*}
\end{defn}\par
\begin{rem}
    Notice that if $C$ has a centre of symmetry, then it holds
     \begin{equation*}
         \psi_C=\displaystyle\max_{s\in\mathbb{T}_{|\partial C|}}\left|\nu_C(s)\right|.
     \end{equation*}
    	Moreover, given an integer $n>1$, if $C$ is the regular polygon with $2n$ sides, then it follows that $\psi_C=(1-n^{-1})\pi$. On the other hand, if $C$ has a $\mathcal{C}^1$ boundary (that is, it has no angular points), then it follows that $\mathcal{T}_C=\varnothing$ and $\psi_C=0$. Last, notice that it always holds $0\leq\psi_C<\pi$.
    \end{rem}
It is time to state our two main results on the affine quadratic discrepancy. The first one shows that for averages over a large enough interval of rotations, we essentially get the same asymptotic order as for a complete rotation.
    \begin{thm}\label{t1}
    	Let $C\subset\mathbb{R}^2$ be a convex body, and let $I\subseteq\mathbb{T}_{2\pi}$ be an interval of angles such that $\psi_C<|I|\leq2\pi$. Then, it holds
    	\begin{equation*}
    	\inf_{\# \mathcal{P}=N}\mathcal{D}_2(\mathcal{P},\, C,\,I) \asymp N^{1/2}.
    	\end{equation*}
    \end{thm}
    Studying the asymptotic behaviour of the Fourier transform of $\mathds{1}_C$ is a fundamental step of the proof, and it is the purpose of Lemma~\ref{CN}. Then, the proof of the lower bound in Theorem~\ref{t1} requires an argument of Cassels \cite{MR0087709} and Montgomery \cite[Ch.~6]{MR1297543} for estimating exponential sums from below, and we present this in Lemma~\ref{C-M}. On the other hand, the upper bound is simple since it just requires unions of uniform lattices.
    \par
	 Our second main result concerns the complementary case of averages over small enough intervals of rotations. 
	\begin{thm}\label{t6}
	Let $C\subset\mathbb{R}^2$ be a convex body, and let the interval $I\subset\mathbb{T}_{2\pi}$ be such that $0<|I|<\psi_C$. It holds
	\begin{equation*}
	\inf_{\# \mathcal{P}=N}\mathcal{D}_2(\mathcal{P},\, C,\,I)\asymp N^{2/5}.
	\end{equation*}
	\end{thm}
 Again, Lemma~\ref{CN} is the starting point for the proof. Then, the proof of the lower bound in Theorem~\ref{t6} relies on an argument in \cite{MR4358540}, and we present it under a general form in Theorem~\ref{t4}. Finally, the proof of the upper bound is more involved than the one in Theorem~\ref{t1} and requires unions of special sets of points that are lattices under certain affine transformations.
 \par
  Finally, in Section~\ref{S4}, we study the intermediate case of $|I|=\psi_C$. Namely, we show that in such circumstances, the affine quadratic discrepancy can achieve any polynomial order in between $N^{1/2}$ and $N^{2/5}$. First, we construct suitable planar convex bodies and establish subtle geometric estimates for their Fourier transform. Then, the main result of the last section, Theorem~\ref{InterTeo}, follows from the aforementioned estimates and by adjusting the arguments in Section~\ref{S3}.

\section{Main results: Asymptotic behaviour of the Fourier transform}

    We introduce a few geometric tools before stating the results whose proofs are in Section~\ref{S2}. First, we give notions on chords and diameters of a body.

    \begin{figure}
        \centering
        \includegraphics[width=0.9\linewidth]{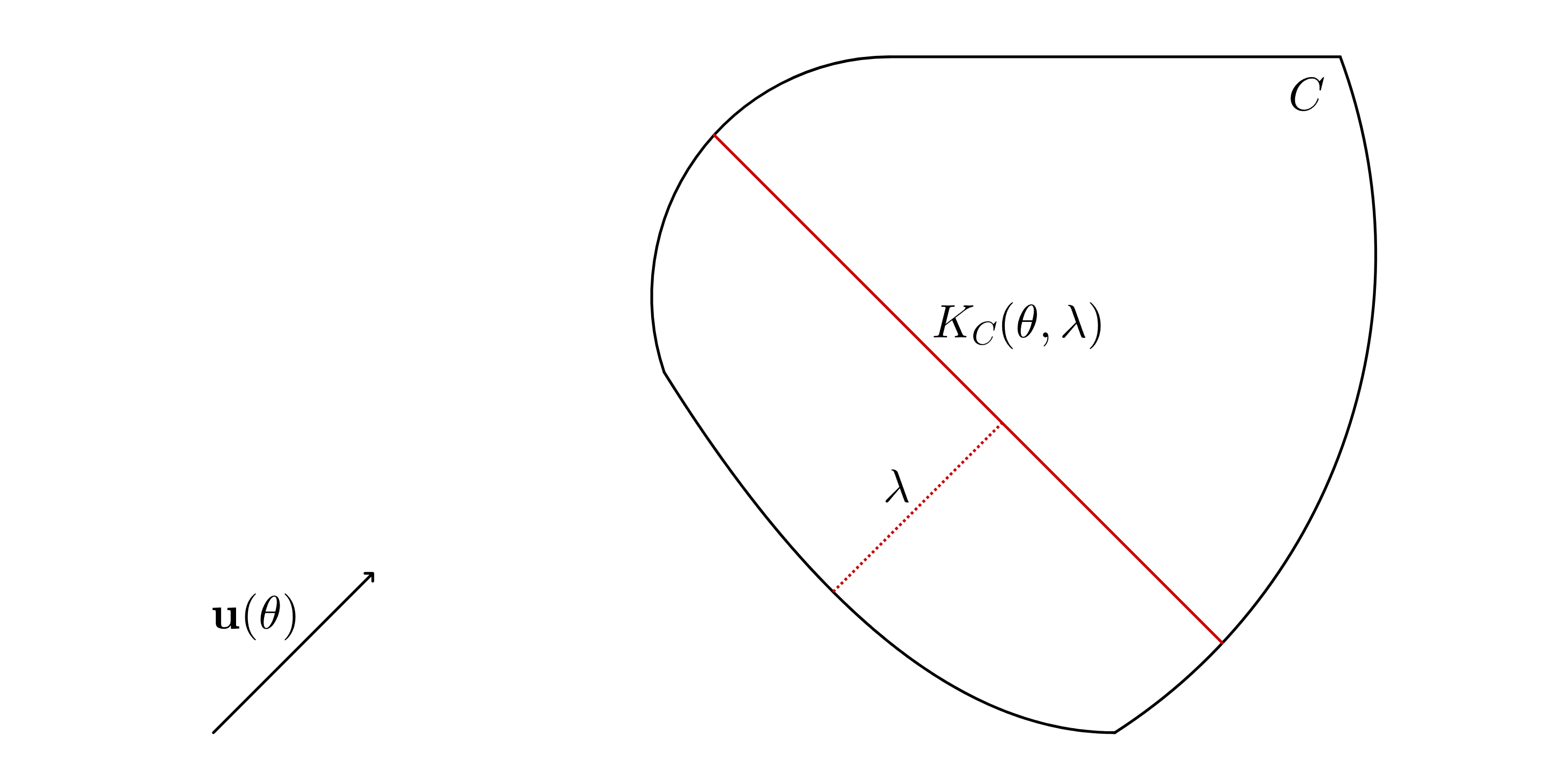}
        \caption{The chord in Definition~\ref{Corde1}.}
        \label{FF1}
    \end{figure}
    
    \begin{defn}\label{Corde1}
		Let $C\subset\mathbb{R}^2$ be a convex body. For an angle $\theta\in\mathbb{T}_{2\pi}$ and value $\lambda>0$, we define the \emph{chord} of $C$ in direction $\uthe$ at distance $\lambda$ as
		\begin{equation*}
		K_C(\theta,\lambda)=\left\{\mathbf{x}\in C\,\colon\,\mathbf{x}\cdot\uthe=\inf_{\mathbf{y}\in C}(\mathbf{y}\cdot\uthe)+\lambda\right\}.
		\end{equation*}
		Further, we consider its length $\left|K_C(\theta,\lambda)\right|$, and we define the quantity
		\begin{equation*}
		{\gamma}_{C}(\theta,\lambda)=\max\{\left|K_C(\theta,\lambda)\right|,\left|K_C(\theta+\pi,\lambda)\right|\}.
		\end{equation*}
		Last, we define the \emph{longest directional diameter} (or classic diameter) of $C$ as
		\begin{equation*}
		L_C=\max_{\mathbf{x},\mathbf{y}\in C}|\mathbf{x}-\mathbf{y}|,
		\end{equation*}
		and we define the \emph{shortest directional diameter} of $C$ as
		\begin{equation*}
		S_C=\min_{\theta\in\mathbb{T}_{2\pi}}\max_{\lambda\geq0}\left|K_C(\theta,\lambda)\right|.
		\end{equation*}
	\end{defn}

    Secondly, we introduce an object that relates directions and perimeter.

    \begin{defn}\label{D4}
        Let $C\subset\mathbb{R}^2$ be a convex body. For an interval of angles $I\subset\mathbb{T}_{2\pi}$, we define the \emph{portion of perimeter} of $C$ with respect to $I$ as

        \begin{equation*}
            P_C(I)=\left|\left\{s\in\mathbb{T}_{\left|\partial C\right|}\,\colon\,\nu_C(s)\cap I\neq\varnothing \right\}\right|.
        \end{equation*}
    \end{defn}

    The following lemma relates the Fourier transform of a planar convex body with its chords, and in particular, it is built upon the one-dimensional results in \cite{MR1166380} and \cite{MR4358540}.
    \begin{lem}\label{t3}
    	There exist positive absolute constants $\kappa_3$ and $\kappa_4$ such that, for every convex body $C\subset\mathbb{R}^2$, for every angle $\theta\in\mathbb{T}_{2\pi}$ and for every value $\rho\geq \kappa_3L_C^{6}/S_C^{7}$, it holds
    	\begin{equation*}
    	\kappa_4\rho^{-2}\gamma^2_{C}(\theta,\rho^{-1})\leq\int_{0}^{1}\left|\widehat{\mathds{1}}_{[\delta] C}(\rho\,\uthe)\right|^2\de \delta\leq2\rho^{-2}\gamma^2_{C}(\theta,\rho^{-1}).
    	\end{equation*}
    \end{lem}

    \begin{rem}
        Notice that if $C$ has a $\mathcal{C}^2$ boundary with curvature that is uniformly bounded away from zero and infinity, then the order in the previous inequality would be of $\rho^{-3}$ uniformly in $\theta$. On the other hand, if $C$ is an axis-symmetric square, then one would find an order of $\rho^{-4}$ at direction $\theta=\pi/4$ and an order of $\rho^{-2}$ at direction $\theta=0$. Last, with some work, one may construct $C$ in such a way that chords in the same direction $\uthe$ display different polynomial decays at different magnitudes of $\rho$; therefore, one would get that the asymptotic behaviour of $\widehat{\mathds{1}}_{C}(\rho\,\uthe)$ oscillates between different polynomial orders at different magnitudes.
    \end{rem}

    When further considering averages over rotations, we find a neat relation between the decay of the Fourier transform of $\mathds{1}_C$ and parts of $\partial C$. Indeed, the starting point for our results on discrepancy is the following.
    
    \begin{lem}\label{CN}
    Uniformly for every convex body $C\subset \mathbb{R}^2$, and uniformly for every closed interval $I\subset\mathbb{T}_{2\pi}$, it holds
	\begin{equation*}
	\int_{I}\int_{0}^{1}\left|\widehat{\mathds{1}}_{[\delta] C}(\rho\,\uthe)\right|^2\de \delta\de \theta\asymp\rho^{-3}\left(P_C\left(I\right)+P_C\left(I+\pi\right)\right),
 \end{equation*}
 with the convention that if $P_C\left(I\right)+P_C\left(I+\pi\right)=0$, then it holds
 \begin{equation*}
     \lim_{\rho\to+\infty}\rho^{3}\int_{I}\int_{0}^{1}\left|\widehat{\mathds{1}}_{[\delta] C}(\rho\,\uthe)\right|^2\de \delta\de \theta=0.
 \end{equation*}
\end{lem}
The latter result is complementary to the estimates of Beck \cite{MR0906524} and Montgomery \cite[Ch.~6]{MR1297543} in the case of complete rotations, and indeed, they both did find a dependence on the perimeter $|\partial C|$. More generally, the problem of estimating the Fourier transform of a geometric body (in arbitrary dimension) has a long history, and as examples, we refer the reader to \cite{Hla50, Herz62, Rand69-1, Ran69-2, BNW88, CDMM90}. In particular, our approach does not involve the Gaussian curvature, as it does not make use of the method of stationary phase for oscillatory integrals.\par
Once taken into account Lemma~\ref{t3}, the proof of Lemma~\ref{CN} relies on Proposition~\ref{t5}, which finds an exact relation between averages over semi-chords of a planar convex body and portions of its perimeter. It is a pivotal point of this paper, but in order to state such a result, we need to expand on our previous notion of chord.

    \begin{figure}
        \centering
        \includegraphics[width=0.9\linewidth]{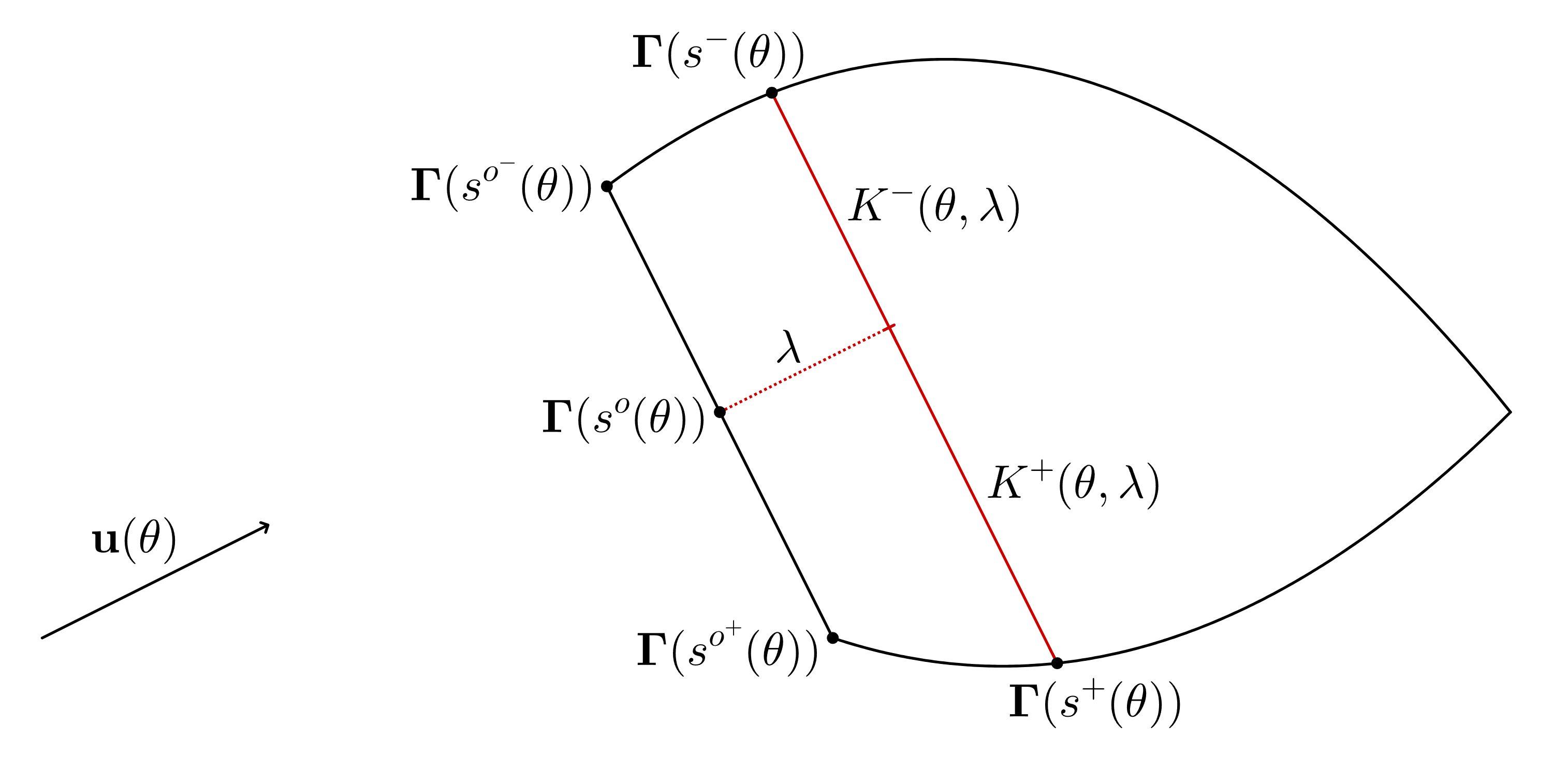}
        \caption{The objects in Definition~\ref{Corde2}. For simplicity, we omit to write $C$.}
        \label{FF2}
    \end{figure}
     
    \begin{defn}\label{Corde2}
        Let $K_C(\theta,\lambda)$ be as in Definition~\ref{Corde1}. We set 
        \begin{equation*}
            s_C^-(\theta,\lambda)\quad\text{and}\quad s_C^+(\theta,\lambda)
        \end{equation*}
        to be the parameterization by $\boldsymbol{\Gamma}_C$ of the extreme points of $K_C(\theta,\lambda)$, with the convention that 
        \begin{equation*}
            \boldsymbol{\Gamma}_C(s_C^-(\theta,\lambda))-\boldsymbol{\Gamma}_C(s_C^+(\theta,\lambda))=\left|K_C(\theta,\lambda)\right|\duthe.
        \end{equation*}
    Further, we define
    \begin{equation*}
        s_C^{o^-}(\theta)=\lim_{\lambda\to0}s_C^-(\theta,\lambda)\quad\text{and}\quad s_C^{o^+}(\theta)=\lim_{\lambda\to0}s_C^+(\theta,\lambda),
    \end{equation*}
    and we set
    \begin{equation*}
    s_C^o(\theta)=s_C^{o^-}(\theta)+\frac{\eta_{|\partial C|}\!\left(s_C^{o^-}(\theta),s_C^{o^+}(\theta)\right)}{2}.
    \end{equation*}
    Hence, we define the \emph{right semi-chord} $K_C^+(\theta,\lambda)$ to be the projection of
    \begin{equation*}
    \boldsymbol{\Gamma}_C\left(\left[s_C^o(\theta),s_C^+(\theta,\lambda)\right]\right) \quad\text{in direction}\quad \uthe \quad\text{on} \quad K_C(\theta,\lambda), 
    \end{equation*}
    and we define $K_C^-(\theta,\lambda)$ analogously.
    \end{defn}
    
    Having gathered all the previous definitions, we are able to state our technical result; in particular, we state it in the case of right semi-chords.
    
    \begin{prop}\label{t5}
	Let $C\subset \mathbb{R}^2$ be a convex body, and let $I=(\alpha,\beta]\subset\mathbb{T}_{2\pi}$ be a left semi-open interval. It holds
	\begin{equation}\label{TeoRota}
	\lim_{\lambda\to0}\frac{1}{2\lambda}\int_I\left|K_C^+(\theta,\lambda)\right|^2\de \theta = P_C(I).
	\end{equation}
\end{prop}

\begin{rem}\label{r1}
    Suppose there exists $s\in\mathbb{T}_{|\partial C|}$ such that
 \begin{equation*}
     \nu_C(s)=[\alpha_1,\beta_1]\subset\mathbb{T}_{2\pi}\quad\text{with}\quad\alpha_1\neq\beta_1.
 \end{equation*}
 By some basic geometry, we get that for every $\theta\in(\alpha_1,\beta_1)$ it holds
	\begin{equation*}
		\lim_{\lambda\to0}\frac{1}{\lambda}\left| K_{C}(\theta,\lambda)\right|=\cot \left(\eta_{2\pi}(\alpha_1,\theta)\right)+\cot \left(\eta_{2\pi}(\theta,\beta_1)\right).
	\end{equation*}
	Hence, if we consider an interval $[\alpha_2,\beta_2]\subset(\alpha_1,\beta_1)$ and set $\rho=1/\lambda$, we get
	\begin{equation*}
 \int_{\alpha_2}^{\beta_2}\left| K_{C}(\theta,\rho^{-1})\right|^2\de \theta\asymp\rho^{-2}.
	\end{equation*}
\end{rem}
Last, if we pair Proposition~\ref{t5} with Lemma~\ref{Pod1}, we get the following spherical estimates on the Fourier transform.
\begin{lem}\label{LC}
    Uniformly for every convex body $C\subset \mathbb{R}^2$, and uniformly for every closed interval $I\subset\mathbb{T}_{2\pi}$, it holds
	\begin{equation*}
	\int_{I}\left|\widehat{\mathds{1}}_{C}(\rho\,\uthe)\right|^2\de \theta\preccurlyeq\rho^{-3}\left(P_C\left(I\right)+P_C\left(I+\pi\right)\right),
 \end{equation*}
 with the convention that if $P_C\left(I\right)+P_C\left(I+\pi\right)=0$, then it holds
 \begin{equation*}
     \lim_{\rho\to+\infty}\rho^{3}\int_{I}\left|\widehat{\mathds{1}}_{C}(\rho\,\uthe)\right|^2\de \theta=0.
 \end{equation*}
\end{lem}

In particular, considering complete rotations in the latter result, we recover an asymptotic version of the spherical estimates in \cite{MR1166380}. Also, we point out that in arbitrary dimensions and for more general sets, the authors in \cite{MR2006553} obtain an analogous decay; it seems plausible that the best constant should (as well) be the perimeter of the set considered, but this is an open matter.

\section{History of the problem}
To better comprehend the frame of this work, we start with a basic definition. As in the Euclidean case, we denote the characteristic function of an interval $I\subset\mathbb{T}$ by $\mathds{1}_I$. In one dimension, a sequence $\left\{p_j\right\}_{j=1}^\infty\subset\mathbb{T}$ is said to be uniformly distributed if for every interval $I\subseteq\mathbb{T}$, it holds 
	\begin{equation*}
 \lim_{N\to+\infty}N^{-1}\sum_{j=1}^{N}\mathds{1}_I(p_j)=|I|.
	\end{equation*}
 The concept of discrepancy has been introduced as a quantitative counterpart to the notion of uniform distribution. Namely, for a positive integer $N$, the discrepancy of a sequence $\mathcal{P}=\{p_j\}_{j=1}^\infty\subset\mathbb{T}$ is defined as
	\begin{equation*}
	D(\mathcal{P},\,N)=\sup_{0<x<1}\left|\sum_{j=1}^{N}\mathds{1}_{[0,x]}(p_j)-Nx\right|.
	\end{equation*}
 In 1935, van der Corput \cite{zbMATH03019347} conjectured that for any sequence $\mathcal{P}\subset\mathbb{T}$, the quantity $D(\mathcal{P},\, N)$ stays unbounded with respect to $N$. Ten years later, the conjecture was proved true by van Aardenne-Ehrenfest \cite{MR0015143, MR32717} with a first lower bound. In 1954, Roth \cite{MR66435} significantly improved the previously established lower bound as a consequence of a result he achieved in the two-dimensional setting. As in the Euclidean case, for a set $\Omega\subset\mathbb{T}^2$ and for a set of $N$ points $\mathcal{P}_N\subset\mathbb{T}^2$, the discrepancy of $\mathcal{P}_N$ with respect to $\Omega$ usually refers to the quantity
\begin{equation*}
    \mathcal{D}(\mathcal{P}_N,\,\Omega)=\sum_{{\mathbf{p}}\in\mathcal{P}_N}\mathds{1}_{\Omega}({\mathbf{p}})-N|\Omega|,
\end{equation*}
where, again, $\mathds{1}_{\Omega}$ stands for the characteristic function of $\Omega$. We state the celebrated result of Roth \cite{MR66435}, in which discrepancy over a family of rectangles is considered.
	\begin{teocit}[Roth]
		It holds
		\begin{equation*}
		\inf_{\# \mathcal{P}=N}\int_{0}^{1}\int_{0}^{1}\left|\mathcal{D}\left(\mathcal{P},\,[0,x)\times[0,y)\right)\right|^2\de x\de y\succcurlyeq\log N.
		\end{equation*}
	\end{teocit}
 The latter is a turning point in discrepancy theory, and the author himself considered it his best work (see \cite{MR3731299} for more historical details). The proof employs the classic orthogonal Haar basis, introducing a new geometric point of view into the field. We refer to \cite{MR2817765} for an extensive survey on the impact of Roth's result. 
 In 1956, H. Davenport \cite{MR0082531} showed that Roth's lower bound could not be improved, therefore proving its sharpness. Later, in 1994, Montgomery \cite[Ch.~6]{MR1297543} introduced an original approach employing Fourier series and obtained the following result.
    \begin{teocit}[Montgomery]
    	It holds
    	\begin{equation*}
    	\inf_{\# \mathcal{P}=N}\int_{0}^{1}\int_{\mathbb{T}^2}\left|\mathcal{D}\left(\mathcal{P},\,\boldsymbol{\tau}+[0,\delta)^2\right)\right|^2\de \boldsymbol{\tau}\de \delta\succcurlyeq\log N.
    	\end{equation*}
    \end{teocit}
    The proof exploits the convolution structure of discrepancy and uses a lower bound of Cassels \cite{MR0087709} for estimating exponential sums. 
    In 1996, Drmota \cite{MR1401710} proved Montgomery's estimate to be sharp since its substantial equivalence to Roth's one.\par 
    By relating continuous and discrete measures, discrepancy theory finds applications in a variety of fields of mathematics, and as examples, we refer the reader to \cite{MR1470456, MR1779341, MR2683232, MR3307667, MR3330354, MR3307692, MR4391422}. Hence, it is natural to replace the rectangles and squares in the previous theorems with more general sets and study which geometric properties come into play.
    Surprisingly, within the family of convex sets, the lower bound for the discrepancy can be much higher than the logarithm. Indeed, already in 1969, Schmidt \cite{MR0245532} showed that the discrepancy of a disc has a polynomial lower bound. Further, Montgomery's result is a quadratic average over translations and dilations, and therefore, it comes naturally to consider the whole class of affine transformations, including rotations.
\par

    In 1988, Beck \cite{MR0906524} got the following major result on the affine quadratic discrepancy in the case of complete rotations.
    \begin{teocit}[Beck]
    	Uniformly for every convex body $C\subset\mathbb{R}^2$, it holds
    	\begin{equation*}
    	\inf_{\# \mathcal{P}=N}\mathcal{D}_2(\mathcal{P},\, C,\,\mathbb{T}_{2\pi})\succcurlyeq |\partial C| N^{1/2},
    	\end{equation*}
     where $|\partial C|$ stands for the perimeter of $C$.
    \end{teocit}
    A few years later, in an independent work, Montgomery \cite[Ch.~6]{MR1297543} obtained a similar result, dropping the hypothesis of convexity but requiring $\partial C$ to be a piecewise-$\mathcal{C}^1$ simple curve. By combining the results of Kendall \cite{MR0024929} and Podkorytov \cite{MR1166380}, the lower bound of Beck and Montgomery turns out to be sharp. Recently, Gennaioli and the author \cite{beretti2024fouriertransformbvfunctions} established a general result on the affine quadratic discrepancy that extends the estimates of Beck and Montgomery to a broad class of bounded variation functions; in particular, our arguments rely on geometric measure-theoretic techniques. Further, we point out that averaging over dilations is necessary and cannot be dropped, as the reader may verify in \cite{MR3500235}. Finally, by substituting $C$ in the previous theorem with a disc and by its invariance under rotations, we get that the quadratic discrepancy of a disc averaged over translations and dilations has a sharp lower bound of order $N^{1/2}$.\par
    The quadratic discrepancy of planar convex bodies averaged over translations and dilations has been widely studied. For example, Drmota \cite{MR1401710} showed that the sharp $\log N$ lower bound holds not only for squares but for the broader family of convex polygons. More recently, Brandolini and Travaglini \cite{MR4358540} gave sharp lower bounds for such quadratic discrepancy on a broad class of planar convex bodies with a piecewise-$\mathcal{C}^2$ boundary. Within the same class of planar convex bodies, they retrieved sharp estimates of all the polynomial orders between $N^{1/2}$ and $N^{2/5}$, which is the same range as in our results.
    \par
    The affine quadratic discrepancy with respect to intervals of rotations was still an open matter. Recently, Bilyk and Mastrianni \cite{MR4585469} got partial results studying the case of a square, and the questions raised thereafter motivated the current work; indeed, we disproof the expectations stated at the end of their paper, where the authors suggested that the affine quadratic discrepancy behaves independently of the interval considered, therefore always as in the case of complete rotations. We also mention that the authors in \cite{MR2861537, MR3453360} investigated the discrepancy of rectangles averaged over sets of (possibly unaccountably many) rotations with empty interiors, and interestingly, the results heavily depend on Diophantine approximation properties.    

    \section{Index of geometric notation}
    \begin{multicols}{2}
    \noindent $\eta_p$ at \eqref{D1}\\
    $\boldsymbol{\Gamma}_C$ at Definition~\ref{D2}\\
    $\nu_C$, $\nu_C^-$, $\nu_C^+$ at Definition~\ref{D2}\\
    $\mathcal{T}_C$ at Definition~\ref{D3}\\
    $\psi_C$ at Definition~\ref{D3}\\
    $K_C$ at Definition~\ref{Corde1}\\
    $\gamma_C$ at Definition~\ref{Corde1}\\
    $L_C$ and $S_C$ at Definition~\ref{Corde1}\\
    $P_C$ at Definition~\ref{D4}\\
    $s_C^-$ and $s_C^+$ at Definition~\ref{Corde2}\\
    $s_C^{o}$, $s_C^{o^-}$, $s_C^{o^+}$ at Definition~\ref{Corde2}\\
    $K_C^-$ and $K_C^+$ at Definition~\ref{Corde2}
    
    \end{multicols}

     \section{Proofs: Asymptotic behaviour of the Fourier transform}\label{S2}
	
	Let us start by exploiting the convolutional structure of \eqref{Discrepancy1}, and show how the Fourier transform comes into play. Let $C\subset\mathbb{R}^2$ be a convex body. Consider $\mu_{\rm L}$ to be the Lebesgue measure on $\mathbb{T}^2$, and for a point $\mathbf{p}\in\mathbb{T}^2$, consider $\mu_{\rm D}(\mathbf{p})$ to be the Dirac delta centered at $\mathbf{p}$. By setting
\begin{equation*}
    \tilde{\mu}=\sum_{\mathbf{p}\in\mathcal{P}_N}\mu_{\rm D}(-\mathbf{p})-N\mu_{\rm L},
\end{equation*}
we get that
\begin{equation*}
\mathcal{D}(\mathcal{P}_N,\, [\boldsymbol{\tau}] C)=\int_{\mathbb{T}^2}\mathfrak{P}\{\mathds{1}_C\}({x-\boldsymbol{\tau}})\de \tilde{\mu}(-{x})=\left(\mathfrak{P}\{\mathds{1}_C\}\ast\tilde{\mu}\right)(\boldsymbol{\tau}).
\end{equation*}
Now, for $f\in L^1(\mathbb{T}^2)$ or $f\in\mathcal{M}(\mathbb{T}^2)$ (that is, the vector space of finite measures on $\mathbb{T}^2$ with values in $\mathbb{R}$), we let 
\begin{equation*}
    {\mathcal{F}}\{f\}\colon \mathbb{Z}^2\to\mathbb{C}
\end{equation*}
be the function of the Fourier coefficients of $f$. In particular, it is not difficult to see that, for every $\mathbf{n}\in\mathbb{Z}^2$, it holds
\begin{equation*}
    {\mathcal{F}}\circ \mathfrak{P}\{\mathds{1}_C\}(\mathbf{n})=\widehat{\mathds{1}}_C(\mathbf{n}).
\end{equation*}
Therefore, by applying Parseval's identity on $\mathbb{T}^2$ and by \eqref{FourierProp} we get
\begin{align*}
    \int_{\mathbb{T}^2}\left|\mathcal{D}(\mathcal{P}_N,\, [\boldsymbol{\tau},\delta,\theta]C)\right|^2\de\boldsymbol{\tau}&=\int_{\mathbb{T}^2}\left|(\mathfrak{P}\{\mathds{1}_{[\delta,\theta]C}\} \ast \tilde{\mu})\right|^2(\boldsymbol{\tau})\de\boldsymbol{\tau}\\
    &=\sum_{{\mathbf{n}}\in\mathbb{Z}^2}\left|{\mathcal{F}}\circ \mathfrak{P}\{\mathds{1}_{[\delta,\theta]C}\}({\mathbf{n}})\right|^2\left|{\mathcal{F}}\{\tilde{\mu}\}({\mathbf{n}})\right|^2\\
    &=\sum_{{\mathbf{n}}\in\mathbb{Z}_*^2}\left|\widehat{\mathds{1}}_{[\delta,\theta]C}({\mathbf{n}})\right|^2\left|\sum_{\mathbf{p}\in\mathcal{P}_N}e^{2 \pi i \mathbf{p}\cdot{\mathbf{n}}}\right|^2,\\
    &=\delta^2\sum_{{\mathbf{n}}\in\mathbb{Z}_*^2}\left|\widehat{\mathds{1}}_C(\delta\sigma_{-\theta}\mathbf{n})\right|^2\left|\sum_{\mathbf{p}\in\mathcal{P}_N}e^{2 \pi i \mathbf{p}\cdot{\mathbf{n}}}\right|^2,
\end{align*}
where, for the sake of notation, we have set $\mathbb{Z}^2_*=\mathbb{Z}^2\setminus\{\mathbf{0}\}$.

		In this first section, we study the asymptotic behaviour of $\widehat{\mathds{1}}_{C}$. Namely, letting $\theta\in\mathbb{T}_{2\pi}$ be an angle and considering $\rho$ to be a real positive number, we are concerned with the decay of
  \begin{equation*}
      \widehat{\mathds{1}}_C(\rho\,\uthe)\quad\text{as}\quad\rho\to+\infty.
  \end{equation*}
  First notice that, since $\mathds{1}_C$ is a real function, it holds
\begin{equation*}
\left|\widehat{\mathds{1}}_{C}(\rho\,\uthe)\right|=\left|\widehat{\mathds{1}}_{C}(\rho\,\mathbf{u}(\theta+\pi))\right|.
\end{equation*}
  Without loss of generality assume $\theta=0$, so that
	\begin{equation*}
	\widehat{\mathds{1}}_C\left((\rho,0)\right)=\int_{\mathbb{R}}\int_{\mathbb{R}}\mathds{1}_{C}(x_1,x_2)e^{-2\pi i \rho x_1}\de x_1\de x_2=\int_{\mathbb{R}}g(x_1)e^{-2\pi i \rho x_1}\de x_1=\widehat{g}(\rho),
	\end{equation*}
	where have set
	\begin{equation}\label{e11}
	g(t)=\int_{\mathbb{R}}\mathds{1}_C(t,x_2)\de x_2.
	\end{equation}
	Since $C$ is convex, the non-negative function $g$ is supported and concave on an interval $[a,b]\subset\mathbb{R}$. Therefore, we are led to study the Fourier transform of such a one-dimensional function, and to proceed, we define an auxiliary tool.
	\begin{defn}
		Let $g:\mathbb{R}\to\mathbb{R}$ be a non-negative function supported and concave on $[a,b]$, then for every $\lambda \in\left[0,\frac{b-a}{2}\right]$ we define the height of g at distance $\lambda$ from the support as
		\begin{equation*}
		\zeta_g(\lambda )=\max\left\{g(a+\lambda ),g(b-\lambda )\right\}.
		\end{equation*}
	\end{defn}
	We remark on the duality between the latter quantity and the chord in Definition~\ref{Corde1}, which is strongly related to the decay of the Fourier transform of $\mathds{1}_C$. It holds the following estimate, obtained through a simple geometric argument. In particular, notice that the threshold and the values involved depend solely on the diameters of $C$.
	\begin{prop}\label{r0}
	Let $C\subset\mathbb{R}^2$ be a convex body. For every $\theta\in\mathbb{T}_{2\pi}$ and for every $\rho\geq2/S_C$, it holds
	\begin{equation*}
	\gamma_C(\theta,\rho^{-1})\geq\frac{S_C}{L_C}\rho^{-1}.
	\end{equation*}
	\end{prop}
\begin{proof}
	Without loss of generality, suppose $\theta=0$ and define $g$ as in \eqref{e11}. In particular, notice that
	\begin{equation*}
	\zeta_g(\rho^{-1})=\gamma_{C}(0,\rho^{-1}),
	\end{equation*}
	so that it is enough to estimate $g$. It is not difficult to see that
	\begin{equation}\label{Triv}
	S_C\leq\max_{x\in\mathbb{R}} g(x)\leq L_C\quad\text{and}\quad S_C\leq|\text{supp}(g)|\leq L_C,
	\end{equation}
	and by the concavity of $g$ on its support, it follows from some easy geometric observations that, for every $\rho\geq2/S_C$, it holds
	\begin{equation*}
	g(\rho^{-1})\geq\frac{\max_{x\in\mathbb{R}} g(x)}{|\text{supp}(g)|}\rho^{-1}\geq\frac{S_C}{L_C}\rho^{-1}.
	\end{equation*}
\end{proof}
   We state a classic upper bound on such one-dimensional functions due to Podkorytov \cite{MR1166380}. For more results in this direction, we refer the interested reader to \cite[Ch.~8]{MR3307692}.
    \begin{lemcit}[Podkorytov]
    	Let $f:\mathbb{R}\to\mathbb{R}$ be a non-negative continuous function supported and concave on $[-1,1]$, then for every value $s\geq1$ it holds
    	\begin{equation*}
    	\left|\widehat{f}(s)\right|\leq s^{-1}\zeta_f(s^{-1}).
    	\end{equation*}
    \end{lemcit}
     Let us show how the latter lemma evolves into estimates on the decay of the Fourier transform of $\mathds{1}_C$. Consider a non-negative function $g:\mathbb{R}\to\mathbb{R}$ supported and concave on a bounded interval $[a,b]\subset\mathbb{R}$, and apply the affine change of variable
    \begin{equation}\label{e9}
    f(s)=g\left(\frac{b+a}{2}+s\frac{b-a}{2}\right),
	\end{equation}
	hence obtaining
	\begin{equation}\label{FourVar}
	\left|\widehat{f}(s)\right|=\frac{2}{b-a}\left|\widehat{g}\left(\frac{2s}{b-a}\right)\right|.
	\end{equation}
    Further, notice that it holds
    \begin{equation*}
    f(\pm(1- \lambda) )=g\left(\frac{b+a}{2}\pm(1- \lambda )\frac{b-a}{2}\right),
    \end{equation*}
    and therefore, for every $\lambda \in\left[0,\frac{b-a}{2}\right]$, we get
    \begin{equation*}
    \zeta_f(\lambda )=\zeta_g\!\left(\lambda \,\frac{b-a}{2}\right).
    \end{equation*}
    Hence, by applying the latter lemma to $f$ and by translating into terms of $g$, we have that, for every $s\geq1$, it holds
    \begin{equation*}
    \frac{2}{b-a}\left|\widehat{g}\left(\frac{2s}{b-a}\right)\right|\leq s^{-1}\zeta_g\!\left(\frac{b-a}{2s}\right),
    \end{equation*}
    so that by the change of variable
    \begin{equation*}
    \rho=2s/(b-a),
    \end{equation*}
    we get that, for every $\rho\geq 2/(b-a)$, it holds
    \begin{equation*}
    \left|\widehat{g}(\rho)\right|\leq\rho^{-1}\zeta_g\left(\rho^{-1}\right).
    \end{equation*}
    In particular, we remark that $|b-a|$ is bounded from below by $S_C$ independently on the choice of $\theta$, and therefore, by turning into terms of the convex body $C$, we get the following formulation.
    \begin{lem}\label{Pod1}
			Let $C\subset\mathbb{R}^2$ be a convex body. For every $\theta\in\mathbb{T}_{2\pi}$ and for every $\rho\geq2/S_C$, it holds
			\begin{equation*}
			\left|\widehat{\mathds{1}}_C\left(\rho\,\uthe\right)\right|\leq \rho^{-1}\gamma_{C}(\theta,\rho^{-1}).
			\end{equation*}
    \end{lem}
	We now state an essential result that establishes both a lower and an upper bound on the Fourier transform of one-dimensional functions as the one in \eqref{e11}.
    \begin{lemcit}[Brandolini-Travaglini]
    	There exist positive absolute constants $\kappa_1<1<\kappa_2$ such that, uniformly for every non-negative continuous function $f:\mathbb{R}\to\mathbb{R}$ supported and concave on $[-1,1]$, it holds
    	\begin{equation*}
    	\int_{\kappa_1}^{ \kappa_2}\left|\widehat{f}(\delta s)\right|^2\de \delta\asymp s^{-2}\zeta^2_f( s^{-1}).
    	\end{equation*}
    \end{lemcit}
	Actually, it was Podkorytov who first achieved the latter estimate and then showed it to Travaglini during a personal communication in 2001, but the original proof has never been published. The authors in \cite[Lem.~23]{MR4358540} give an original proof by relating the Fourier transform of such $f$ with its moduli of smoothness (see \cite[Ch.~2]{MR1261635}). Now, with Proposition~\ref{r0} and Lemma~\ref{Pod1} in mind, we turn this result into estimates for the Fourier transform of $\mathds{1}_C$.
    \begin{proof}[Proof of Lemma~\ref{t3}]
    	Let us start by proving the upper bound. First, we set $\rho_0=2/S_C$ and consider $\rho\geq\rho_0$. Then, it is useful to split the integral as
   	\begin{equation}\label{e10}
    \begin{split}&\int_{0}^{1}\left|\widehat{\mathds{1}}_{[\delta] C}(\rho\,\uthe)\right|^2\de \delta=\\
    &=\int_{0}^{\rho_0/\rho}\left|\widehat{\mathds{1}}_{[\delta] C}(\rho\,\uthe)\right|^2\de \delta\,+\,\int_{\rho_0/\rho}^{1}\left|\widehat{\mathds{1}}_{[\delta] C}(\rho\,\uthe)\right|^2\de \delta.
    \end{split}
    \end{equation}
   	By basic properties of the Fourier transform and the fact that $|C|\leq L_C^2$ (this easily follows by \eqref{Triv}), we obtain
   	\begin{equation*}
   	\left\|\widehat{\mathds{1}}_{[\delta] C}\right\|_{L^\infty(\mathbb{R}^2)}\leq\left\|\mathds{1}_{[\delta] C}\right\|_{L^1(\mathbb{R}^2)}=\delta^2\left|C\right|\leq\delta^2L_C^2,
    \end{equation*}
   	so that, for the first integral in the right-hand term of \eqref{e10}, we get
   	\begin{equation}\label{e12}
   	\int_{0}^{\rho_0/\rho}\left|\widehat{\mathds{1}}_{[\delta] C}(\rho\,\uthe)\right|^2\de \delta\leq L_C^4\int_{0}^{\rho_0/\rho}\delta^4\de \delta=\frac{32 L_C^4}{5 S_C^{5}}\rho^{-5}.
    \end{equation}
   	Now, notice that by the concavity of $\left|K_C(\theta,\cdot)\right|$ on its support, we have that, for every angle $\theta\in\mathbb{T}_{2\pi}$, for every $\rho>0$, and for every $\delta\in (0,1]$, it holds
   	\begin{equation*}
   		\gamma_C(\theta,\delta^{-1}\rho^{-1})\leq \delta^{-1}\,\gamma_C(\theta,\rho^{-1}).
   	\end{equation*}
    	Therefore, by the latter observation, and by \eqref{FourierProp} and Lemma~\ref{Pod1}, for the second integral in the right-hand term of \eqref{e10} we get
    	\begin{equation*}
    	\begin{split}
    		\int_{\rho_0/\rho}^{1}\left|\widehat{\mathds{1}}_{[\delta] C}(\rho\,\uthe)\right|^2\de \delta&=\int_{\rho_0/\rho}^{1}\delta^4\left|\widehat{\mathds{1}}_{C}(\delta\rho\,\uthe)\right|^2\de \delta\\
    		&\leq\int_{\rho_0/\rho}^{1}\delta^4\left|\delta^{-1}\rho^{-1}\gamma_C(\theta,\delta^{-1}\rho^{-1})\right|^2\de \delta\\
    		&\leq\int_{\rho_0/\rho}^{1}\delta^2\left|\rho^{-1}\delta^{-1}\gamma_C(\theta,\rho^{-1})\right|^2\de \delta\leq\rho^{-2}\gamma_C^2(\theta,\rho^{-1}).
    	\end{split}
    	\end{equation*}
    	By Proposition~\ref{r0}, it holds
    	\begin{equation*}
    	\rho^{-2}\gamma_{C}^2(\theta,\rho^{-1})\geq \frac{S_C^2}{L_C^{2}}\rho^{-4},
    	\end{equation*}
    	so that, by defining
    	\begin{equation*}
    	\rho_1=\frac{32 L_C^6}{5 S_C^{7}},
    	\end{equation*}
    	one can deduce from \eqref{e12} that, for every $\rho\geq\rho_1$, it holds
    	\begin{equation*}
    	\int_{0}^{\rho_0/\rho}\left|\widehat{\mathds{1}}_{[\delta] C}(\rho\,\uthe)\right|^2\de \delta\leq\frac{S_C^2}{L_C^2}\rho^{-4}\leq \rho^{-2}\gamma_C^2(\theta,\rho^{-1}).
    	\end{equation*}
    	Finally, by combining the latter observations into \eqref{e10}, we obtain that, for every $\rho\geq\max\{\rho_0,\rho_1\}$, it holds
    	\begin{equation*}
    	\begin{split}
    	\int_{0}^{1}\left|\widehat{\mathds{1}}_{[\delta] C}(\rho\,\uthe)\right|^2d\delta&\leq\int_{0}^{\rho_0/\rho}\left|\widehat{\mathds{1}}_{[\delta] C}(\rho\,\uthe)\right|^2d\delta+\rho^{-2}\gamma_C^2(\theta,\rho^{-1})\\
    	&\leq2\rho^{-2}\gamma_C^2(\theta,\rho^{-1}).
  	    \end{split}
    	\end{equation*}

    	Let us now proceed to prove the lower bound. As before, and without loss of generality, we assume $\theta=0$, and we define $g$ as in \eqref{e11}.
    	Hence, we define $f$ by the same affine change of variable as in \eqref{e9}, so that its support is the interval $(-1,1)$. By the latter lemma, it follows that there exist positive absolute constants $\tilde{s}>1$ and $\tilde{c}>0$ such that, uniformly for every such $f$ and for every $s\geq \tilde{s}$, it holds
    	\begin{equation*}
    	s^{-2}\zeta^2_f( s^{-1})\leq \tilde{c}\int_{\kappa_1}^{ \kappa_2}\left|\widehat{f}(\delta s)\right|^2\de \delta.
    	\end{equation*}
    	By the concavity of $f$ on its support, it follows that, for every $s_1$ and $s_2$ such that $0\leq s_1<s_2\leq1$, it holds
    	\begin{equation*}
    	f(-1+s_1)\leq2f(-1+s_2)\quad\text{and}\quad f(1-s_1)\leq2f(1-s_2).
    	\end{equation*}
    	Hence, since $ \kappa_2> 1$ and $\tilde{s}>1$, then for every $s\geq \tilde{s}$ it holds
    	\begin{equation*}
    s^{-2}\zeta^2_f( \kappa_2^{-1} s^{-1})\leq  4 \tilde{c}\int_{\kappa_1}^{ \kappa_2}\left|\widehat{f}(\delta s)\right|^2\de \delta\leq \frac{4\tilde{c}}{\kappa_1^2}\int_{\kappa_1}^{\kappa_2}\delta^2 \left|\widehat{f}(\delta s)\right|^2\de \delta.
    	\end{equation*}
    	Turning into terms of $g$, and by \eqref{FourVar} and the change of variable, $\rho=2 s \kappa_2/(b-a)$, we get that, for every $\rho$ such that $\rho\geq2 \tilde{s} \kappa_2 /(b-a)$, it holds
    	\begin{equation*}
    	\kappa_2^2\rho^{-2}\zeta^2_g\left(\rho^{-1}\right)\leq \frac{4\tilde{c}}{\kappa_1^2} \int_{\kappa_1}^{ \kappa_2}\delta^2\left|\widehat{g}\left(\delta \kappa_2^{-1}\rho\right)\right|^2d\delta.
        \end{equation*}
        Independently of the choice of $\theta$, it holds $\left|b-a\right|\geq S_C$, and then we set
        \begin{equation*}
            \rho_2=2 \tilde{s} \kappa_2 /S_C.
        \end{equation*}
        Hence, by rewriting the last inequality in terms of $C$, and by the change of variable $\delta=\kappa_2\Delta$, we get that for every $\rho\geq\rho_2$ it holds
        \begin{equation*}
        \rho^{-2}\gamma_C^2(0,\rho^{-1})\leq \frac{4\tilde{c}\kappa_2}{\kappa_1^2} \int_{0}^{1}\Delta^2\left|\widehat{\mathds{1}}_C(\Delta\rho,0)\right|^2\de \Delta=\frac{4\tilde{c}\kappa_2}{\kappa_1^2} \int_{0}^{1}\left|\widehat{\mathds{1}}_{[\Delta]C}(\rho,0)\right|^2\de \Delta.
        \end{equation*}
        Last, we set $\kappa_4=\kappa_1^2/(4\tilde{c}\kappa_2)$, and the conclusion follows once we acknowledge that there exists a positive absolute constant $\kappa_3$, independent of $C$, such that it holds
        \begin{equation*}
        \max\{\rho_0,\rho_1,\rho_2\}\leq \kappa_3L_C^6/S_C^{7}.
        \end{equation*}
    \end{proof}

	\begin{rem}
		Notice that the estimates in the latter lemma are uniform for a class of planar convex bodies whose longest and shortest directional diameters are uniformly bounded.
	\end{rem}

We proceed with the proof of Proposition~\ref{t5}, which is indeed the tool that allows us to study averages over intervals of rotations.

\begin{proof}[Proof of Proposition~\ref{t5}]
	For the sake of simplicity, we omit the subscript $C$ under the geometric objects. With the help of Figure~\ref{FF3}, observe that
	\begin{equation}\label{l1}
		\left|K^+(\theta,\lambda)\right|=-\int_{s^o(\theta)}^{s^+(\theta,\lambda)}\duthe\cdot\boldsymbol{\Gamma}'(t)\de t,
	\end{equation}
	and
	\begin{equation}\label{l2}
		\int_{s^o(\theta)}^{s^+(\theta,\lambda)}\uthe\cdot\boldsymbol{\Gamma}'(t)\de t=\lambda.
	\end{equation}\begin{figure}
        \centering
        \includegraphics[width=0.9\linewidth]{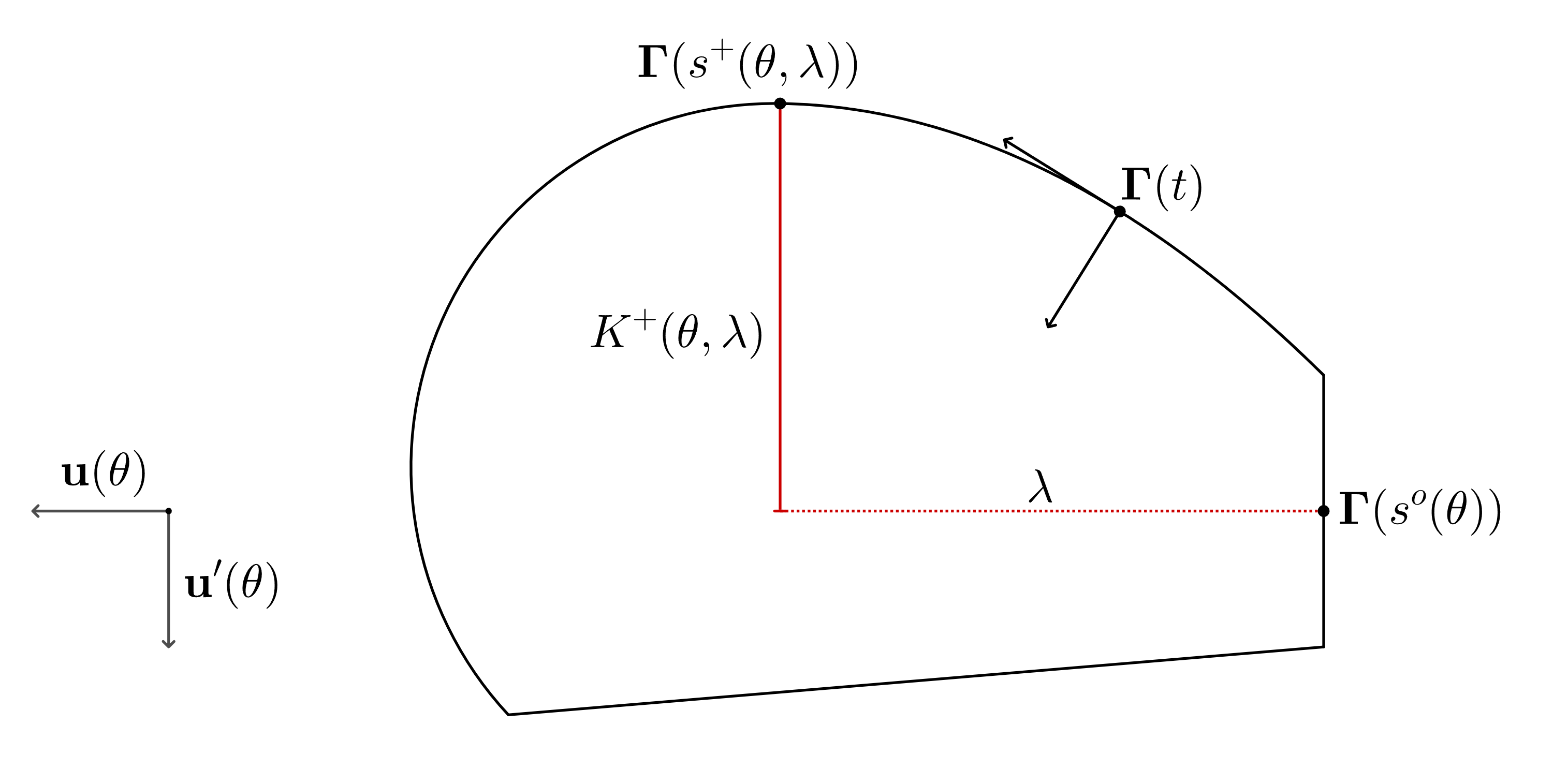}
        \caption{An auxiliary image for the proof of Proposition~\ref{t5}. For simplicity, we omit to write $C$.}
        \label{FF3}
    \end{figure}Since $C$ is a convex body, it is not difficult to deduce that the set of angular points of $C$ is at most countable. In turn, this implies that the derivatives
    \begin{equation*}
        \frac{\partial}{\partial\lambda}s^+,\quad\frac{\partial}{\partial\theta}s^+,\quad \text{and}\quad\boldsymbol{\Gamma}',\quad\text{exist almost everywhere}.
    \end{equation*}
    Hence, by taking the distributional derivative with respect to $\lambda$ of both sides of \eqref{l2}, we get
	\begin{equation}\label{l3}
		\left(\frac{\partial}{\partial\lambda}s^+(\theta,\lambda)\right)\uthe\cdot\boldsymbol{\Gamma}'(s^+(\theta,\lambda))=1.
	\end{equation}
    Also, by taking the distributional derivative with respect to $\theta$ of both sides of \eqref{l2} and by applying Leibniz integral rule, we obtain
	\begin{equation}\label{l6}
 \begin{split}
		&\left(\frac{\partial}{\partial\theta}s^+(\theta,\lambda)\right)\uthe\cdot\boldsymbol{\Gamma}'(s^+(\theta,\lambda))+\int_{s^o(\theta)}^{s^+(\theta,\lambda)}\duthe\cdot\boldsymbol{\Gamma}'(t)\de t=\\
  &=\left(\frac{\partial}{\partial\theta}s^o(\theta)\right)\uthe\cdot\boldsymbol{\Gamma}'(s^o(\theta)).
  \end{split}
	\end{equation}
    It is simple to notice that, for every $\theta\in\mathcal{T}$, it holds
    \begin{equation*}
        \frac{\partial}{\partial\theta}s^o(\theta)=0.
    \end{equation*}
    On the other hand, for every $\theta\in\mathcal{T}^\mathsf{c}$, it holds
    \begin{equation*}
        \uthe\cdot\boldsymbol{\Gamma}'(s^o(\theta))=0.
    \end{equation*}
    Therefore, for every angle $\theta\in\mathbb{T}_{2\pi}$, it holds 
	\begin{equation}\label{l7}
	\left(\frac{\partial}{\partial\theta}s^o(\theta)\right)\uthe\cdot\boldsymbol{\Gamma}'(s^o(\theta))=0.
	\end{equation}
    Hence, by \eqref{l1}, \eqref{l6}, and \eqref{l7}, it follows that
	\begin{equation}\label{l4}
		\left|K^+(\theta,\lambda)\right|=\left(\frac{\partial}{\partial\theta}s^+(\theta,\lambda)\right)\uthe\cdot\boldsymbol{\Gamma}'(s^+(\theta,\lambda)).
	\end{equation}
    Also, by taking the distributional derivative with respect to $\lambda$ of both sides of \eqref{l1}, we get
    \begin{equation}\label{l5}
    	\frac{\partial}{\partial\lambda}\left|K^+(\theta,\lambda)\right|=-\left(\frac{\partial}{\partial\lambda}s^+(\theta,\lambda)\right)\duthe\cdot\boldsymbol{\Gamma}'(s^+(\theta,\lambda)).
    \end{equation}
	Then, since we can apply the dominated convergence theorem to the integral at the left-hand side of \eqref{TeoRota}, and by \eqref{l3},\eqref{l4}, and \eqref{l5}, it follows that
	\begin{align*}
		&\frac{\partial}{\partial\lambda}\int_I\left|K^+(\theta,\lambda)\right|^2\de \theta=\\&=2\int_I\left|K^+(\theta,\lambda)\right|\left(\frac{\partial}{\partial\lambda}\left|K^+(\theta,\lambda)\right|\right)\de \theta\\
		&=-2\int_I\left(\frac{\partial}{\partial\theta}s^+(\theta,\lambda)\right)\uthe\cdot\boldsymbol{\Gamma}'(s^+(\theta,\lambda))\left(\frac{\partial}{\partial\lambda}s^+(\theta,\lambda)\right)\duthe\cdot\boldsymbol{\Gamma}'(s^+(\theta,\lambda))\de \theta\\
		&=-2\int_I\left(\frac{\partial}{\partial\theta}s^+(\theta,\lambda)\right)\duthe\cdot\boldsymbol{\Gamma}'(s^+(\theta,\lambda))\de \theta.
	\end{align*}
    Now, notice that $s^{o^+}(\theta)=s^o(\theta)$ if and only if
    \begin{equation*}
        \left\{ \mathbf{b}\in C\,\colon\, \mathbf{b}\cdot\uthe=\min_{\mathbf{a}\in C}\mathbf{a}\cdot\uthe\right\}\quad\text{is a single point.}
    \end{equation*}
    Also, it is not difficult to see that, uniformly in $\theta\in\mathbb{T}_{2\pi}$, it holds
	\begin{equation*}
	\lim_{\lambda\to0}\boldsymbol{\Gamma}'(s^+(\theta,\lambda))=-\mathbf{u}'(\nu^+(s^{o^+}(\theta))),
	\end{equation*}
	and by the compactness of $\mathbb{T}_{2\pi}$, this in turn implies that for every small $\varepsilon>0$ there exists $\lambda_\varepsilon>0$ such that, for every $\lambda\in\mathbb{R}$ such that $0<\lambda\leq\lambda_\varepsilon$, and uniformly for every angle $\theta\in\mathbb{T}_{2\pi}$, it holds
	\begin{equation*}
	\left|\boldsymbol{\Gamma}'(s^+(\theta,\lambda))+\mathbf{u}'(\nu^+(s^{o^+}(\theta)))\right|<\varepsilon.
	\end{equation*}
    Now, consider the set
    \begin{equation*}
        E_\varepsilon=\left\{\theta\in I\,\colon\,\eta_{2\pi}\!\left(\theta,\nu^+(s^{o^+}(\theta))\right)\geq\varepsilon\right\},
    \end{equation*}
    and let $\left[\alpha_j,\beta_j\right]$ be one of its connected components; in particular, notice that these are at most $2\pi/\varepsilon$. By the fact that $s^o(\alpha_j)=s^o(\beta_j)$, and by some basic geometry, we get that
	\begin{equation*}
	\begin{split}
	\int_{\alpha_j}^{\beta_j}\frac{\partial}{\partial\theta}s^+(\theta,\lambda)\de \theta&=\eta_{|\partial C|}\!\left(s^+(\alpha_j,\lambda),s^+(\beta_j,\lambda)\right)\\
	&\leq\eta_{|\partial C|}\!\left(s^o(\beta_j),s^+(\beta_j,\lambda)\right)\leq\frac{\lambda}{\tan(\varepsilon)}=\lambda\,\mathcal{O}(\varepsilon^{-1}).
    \end{split}
	\end{equation*}
	Moreover, notice that, for every $\theta\in I\setminus E_\varepsilon$, it holds
	\begin{equation*}
	\duthe\cdot\mathbf{u}'(\nu^+(s^{o^+}(\theta)))=\cos(\theta-\nu^+(s^{o^+}(\theta)))\leq\cos(\varepsilon)=1+\mathcal{O}(\varepsilon^2).
	\end{equation*}
    By the latter observations, for every $\lambda$ such that $0<\lambda\leq\lambda_\varepsilon$, it follows that
	\begin{align*}
		&-\int_I\left(\frac{\partial}{\partial\theta}s^+(\theta,\lambda)\right)\duthe\cdot\boldsymbol{\Gamma}'(s^+(\theta,\lambda))\de \theta=\\
		&=\int_I\left(\frac{\partial}{\partial\theta}s^+(\theta,\lambda)\right)\duthe\cdot\mathbf{u}'(\nu^+(s^o(\theta)))\de \theta+\mathcal{O}(\varepsilon)\\
	   	&=\int_{I\setminus E_\varepsilon}\left(\frac{\partial}{\partial\theta}s^+(\theta,\lambda)\right)\duthe\cdot\mathbf{u}'(\nu^+(s^o(\theta)))\de \theta + \lambda\,\mathcal{O}(\varepsilon^{-2})+\mathcal{O}(\varepsilon)\\
    	&=\int_{I\setminus E_\varepsilon}\frac{\partial}{\partial\theta}s^+(\theta,\lambda)\de \theta + \lambda\,\mathcal{O}(\varepsilon^{-2})+ \mathcal{O}(\varepsilon)\\
		&=\int_{I}\frac{\partial}{\partial\theta}s^+(\theta,\lambda)\de \theta + \lambda\,\mathcal{O}(\varepsilon^{-2}) + \mathcal{O}(\varepsilon)\\
        &=\eta_{|\partial C|}\!\left(s^+(\beta,\lambda),s^+(\alpha,\lambda)\right) + \lambda\,\mathcal{O}(\varepsilon^{-2})+ \mathcal{O}(\varepsilon).
	\end{align*}
	Finally, we notice that
	\begin{equation*}
	\lim_{\lambda\to0}\eta_{|\partial C|}\!\left(s^+(\beta,\lambda),s^+(\alpha,\lambda)\right)=P_C((\alpha,\beta]),
	\end{equation*}
	and therefore, by choosing $\lambda=\min\left(\lambda_\varepsilon,\varepsilon^3\right)$ and letting $\varepsilon\to0$, we get that
 \begin{equation*}
     \lim_{\lambda\to 0}\frac{\partial}{\partial\lambda}\int_I\left|K^+(\theta,\lambda)\right|^2\de \theta= 2 P_C((\alpha,\beta]).
 \end{equation*}
 Last, the claim follows at once by applying L'Hospital's rule.
\end{proof}

By an analogous proof, the same result for $K_C^-$ and right semi-open intervals $I=[a,b)$ holds. As for entire chords $K_C$, by the fact that for every $a,b\geq0$ it holds
\begin{equation*}
    \frac{a^2+b^2}{2}\leq\max\left(a^2,b^2\right)\leq a^2+b^2\quad\text{and}\quad a^2+b^2\leq(a+b)^2\leq2a^2+2b^2,
\end{equation*}
it easily follows a handy result.
\begin{cor}\label{c2}
		Let $C\subset \mathbb{R}^2$ be a convex body, and let $I\subset\mathbb{T}_{2\pi}$ be a closed interval. It holds
	\begin{equation*}
 \liminf_{\rho\to+\infty}\rho\int_{I}\gamma^2_C(\theta,\rho^{-1})\de \theta\geq P_C\left(I\right)+P_C\left(I+\pi\right)
    \end{equation*}
    and
    \begin{equation*}
\limsup_{\rho\to+\infty}\rho\int_{I}\gamma^2_C(\theta,\rho^{-1})\de \theta\leq 8 P_C\left(I\right)+8P_C\left(I+\pi\right).
	\end{equation*}
\end{cor}
Pairing the latter result with Lemma~\ref{Pod1} and Lemma~\ref{t3}, we immediately get Lemma~\ref{LC} and Lemma~\ref{CN}. We also retrieve the following.
\begin{lem}\label{lem2}
	Let $C\subset\mathbb{R}^2$ be a convex body, and let $I\subseteq\mathbb{T}_{2\pi}$ be an interval of angles such that $\psi_C<|I|\leq2\pi$. Uniformly for every $\omega\in\mathbb{T}_{2\pi}$, it holds
	\begin{equation*}
	\int_I\int_{0}^1\left|\widehat{\mathds{1}}_{[\delta, \theta] C}(\rho\,{\mathbf{u}(\omega)})\right|^2\de \delta\de \theta\asymp\rho^{-3}.
	\end{equation*}
\end{lem}
\begin{proof}
First, we prove that there exists a positive value $c$ such that for every $\omega\in\mathbb{T}_{2\pi}$ it holds
\begin{equation*}
    \liminf_{\rho\to+\infty}\rho\int_{\omega+I}\gamma^2_C(\theta,\rho^{-1})\de \theta\geq c.
\end{equation*}
If this were not the case, then, by the latter corollary, we would have a sequence of $\{\omega_j\}_{j\in\mathbb{N}}\subset\mathbb{T}_{2\pi}$ such that
\begin{equation*}
    \lim_{j\to+\infty}\left(P_C(\omega_j+I)+P_C(\omega_j+I+\pi)\right)=0.
\end{equation*}
Hence, by the compactness of $\mathbb{T}_{2\pi}$, we would get the existence of a $\tilde{\omega}\in\mathbb{T}_{2\pi}$ such that
\begin{equation*}
    P_C(\tilde{\omega}+I)=0=P_C(\tilde{\omega}+I+\pi),
\end{equation*}
but this is a contradiction since it implies that
\begin{equation*}
    (\tilde{\omega}+I)\cup(\tilde{\omega}+I+\pi)\subset\mathcal{T}_C,
\end{equation*}
and consequently, it would hold $\psi_C\geq |I|$.\par
Finally, by Lemma~\ref{t3} and by the compactness of $\mathbb{T}_{2\pi}$, it follows that, uniformly for every $\omega\in\mathbb{T}_{2\pi}$, it holds
	\begin{equation*}
	\int_I\int_{0}^1\left|\widehat{\mathds{1}}_{[\delta] C}(\rho\,\mathbf{u}(\omega-\theta))\right|^2\de \delta\de \theta\asymp\rho^{-2}\int_{\omega+I}\gamma_C^2(\theta,\rho^{-1})\de \theta\asymp\rho^{-3}.
	\end{equation*}
\end{proof}

\section{Proofs: Discrepancy over Affine Transformations}\label{S3}

We show a classical technical result on estimating exponential sums from below. We also point out that an analogous result holds on manifolds, as recently presented in \cite{MR4238281} and \cite{MR4480211}. Last, the argument in the proof of the following lemma is due to Siegel \cite{MR1555407}.
\begin{lem}[Cassels-Montgomery]\label{C-M}
	Let $U\subset\mathbb{R}^2$ be a neighbourhood of the origin. There exists a positive value $c_U$ such that, for every origin-symmetric convex body $\Omega\subset\mathbb{R}^2$ and for every finite set of points $\{\mathbf{p}_j\}_{j=1}^N\subset\mathbb{T}^2$, it holds
	\begin{equation*}
	\sum_{\mathbf{m}\in(\Omega\setminus U)\cap\mathbb{Z}^2}\left|\sum_{j=1}^{N}e^{2\pi i \mathbf{m}\cdot \mathbf{p}_j}\right|^2\geq\frac{|\Omega|}{4}N-c_U N^2.
	\end{equation*}
\end{lem}

\begin{proof}
	Consider the auxiliary sets $A_\Omega(\mathbf{x})\subset\mathbb{Z}^2$ defined by
	\begin{equation*}
	A_\Omega(\mathbf{x})=\left(\mathbf{x}+\Omega/2\right)\cap \mathbb{Z}^2.
	\end{equation*}
	Notice that
	\begin{equation*}
	\int_{\mathbb{T}^2}\#A_\Omega(\mathbf{x})\de \mathbf{x}=\int_{\mathbb{T}^2}\sum_{\mathbf{n}\in\mathbb{Z}^2}\mathds{1}_{\Omega/2}(\mathbf{n}-\mathbf{x})\de \mathbf{x}=\int_{\mathbb{R}^2}\mathds{1}_{\Omega/2}(\mathbf{x})\de \mathbf{x}=\frac{|\Omega|}{4},
	\end{equation*}
	and therefore, we can individuate a point 
 \begin{equation*}
     \mathbf{x}_*\in[0,1)^2\quad\text{such that}\quad\#A_\Omega(\mathbf{x}_*)\geq\frac{|\Omega|}{4}.
 \end{equation*}
 Hence, consider the non-negative trigonometric polynomial
	\begin{equation*}
	T(\mathbf{y})=\frac{1}{\#A_\Omega(\mathbf{x}_*)}\left|\sum_{\mathbf{n}\in A_\Omega(\mathbf{x}_*)}e^{2\pi i \mathbf{n}\cdot \mathbf{y}}\right|^2=\frac{1}{\#A_\Omega(\mathbf{x}_*)}\sum_{\mathbf{n},\mathbf{m}\in A_\Omega(\mathbf{x}_*)}e^{2\pi i(\mathbf{n}-\mathbf{m})\cdot \mathbf{y}},
	\end{equation*}
	and notice that the function of its Fourier coefficients $\widehat T\colon\mathbb{Z}^2\to\mathbb{R}$ is non-negative as well and, since $\mathbf{n},\mathbf{m}\in A_\Omega(\mathbf{x}_*)$ imply $(\mathbf{n}-\mathbf{m})\in\Omega$, then its support is contained in $\Omega$. Further, observe that we have
	\begin{equation*}
	T(0)=\#A_\Omega(\mathbf{x}_*)\geq \frac{|\Omega|}{4}.
	\end{equation*}
 Since for every $\mathbf{n}\in\mathbb{Z}^2$ it holds
	\begin{equation*}
	0\leq\widehat T(\mathbf{n})\leq\widehat{T}(\mathbf{0})=\int_{\mathbb{T}^2}T(\mathbf{x})\de \mathbf{x}=1,
	\end{equation*}
	then it follows that
	\begin{equation*}
	\begin{split}
	\sum_{\mathbf{n}\in\Omega\cap\mathbb{Z}^2}\left|\sum_{j=1}^{N}e^{2\pi i \mathbf{n}\cdot \mathbf{p}_j}\right|^2&\geq \sum_{\mathbf{n}\in\Omega\cap\mathbb{Z}^2}\widehat{T}(\mathbf{n})\left|\sum_{j=1}^{N}e^{2\pi i \mathbf{n}\cdot \mathbf{p}_j}\right|^2\\
	&=\sum_{j=1}^{N}\sum_{\ell=1}^{N}\sum_{\mathbf{n}\in\Omega\cap\mathbb{Z}^2}\widehat{T}(\mathbf{n})\,e^{2\pi i \mathbf{n}\cdot(\mathbf{p}_j-\mathbf{p}_\ell)}\\
	&=\sum_{j=1}^{N}\sum_{\ell=1}^{N}T(\mathbf{p}_j-\mathbf{p}_\ell)\geq N\frac{|\Omega|}{4}.
	\end{split}
	\end{equation*}
	Last, we get
	\begin{equation*}
	\begin{split}
		\sum_{\mathbf{n}\in\left(\Omega\setminus U\in\mathbb{Z}^2\right)}\left|\sum_{j=1}^{N}e^{2\pi i \mathbf{n} \cdot \mathbf{p}_j}\right|^2&=\sum_{\mathbf{n}\in\Omega\cap\mathbb{Z}^2}\left|\sum_{j=1}^{N}e^{2\pi i \mathbf{n} \cdot \mathbf{p}_j}\right|^2-\sum_{\mathbf{n}\in U\cap\mathbb{Z}^2}\left|\sum_{j=1}^{N}e^{2\pi i \mathbf{n} \cdot \mathbf{p}_j}\right|^2\\
		&\geq N\frac{|\Omega|}{4}-\sum_{\mathbf{n}\in U\cap\mathbb{Z}^2}N^2=  N\frac{|\Omega|}{4}-c_U N^2.
	\end{split}
	\end{equation*}
\end{proof}

We now prove a general result that allows us to obtain lower bounds for the quadratic discrepancy. The original argument is in \cite{MR4358540}, and we present an integral version of the proof. As further notation, we consider the argument function
\begin{equation*}
    \arg\colon\mathbb{R}^2\setminus\{\mathbf{0}\}\to\left(-\frac{\pi}{2},\frac{\pi}{2}\right]\quad\text{defined as}\quad\arg (x_1,x_2)=\arctan\frac{x_2}{x_1}.
\end{equation*}
\begin{thm}\label{t4}
	Let $C\subset\mathbb{R}^2$ be a convex body. Let $\Xi$ be a generic set of transformations of $C$ and let $h\in[0,1]$. If there exist an interval of angles $I\subset\mathbb{T}_{2\pi}$ and values $\tilde{\rho},\tilde{c}>0$ such that for every $\rho\geq\tilde{\rho}$ it holds
	\begin{equation*}
	\int_\Xi\left|\widehat{\mathds{1}}_{[\xi] C}(\rho\,{\mathbf{u}(\omega)})\right|^2\de \xi\geq \tilde{c} \begin{cases}
		\rho^{-3} &\text{if } \omega\in I\cup(I+\pi)\\
		\rho^{-3-h} &\text{else}
	\end{cases},
	\end{equation*}
	then it holds
	\begin{equation*}
	\inf_{\#\mathcal{P}_N=N}\int_\Xi\int_{\mathbb{T}^2}\left|\mathcal{D}(\mathcal{P}_N,\, [\boldsymbol{\tau},\xi] C)\right|^2\de \boldsymbol{\tau}\de\xi\succcurlyeq
	N^{\frac{2}{4+h}}.
	\end{equation*}
\end{thm}
\begin{proof}
    In what follows, we will make some reasonable assumptions so as not to get into tedious (but basic) geometric details and to better convey the ideas of the proof.\par
    Let $N\in\mathbb{N}\setminus\{0\}$. Consider a rectangle $R\subset\mathbb{R}^2$ such that it is symmetric with respect to the axes and has a vertex in $(X,Y)$ where $X=X(N)$ and $Y=Y(N)$ are positive parameters of $N$ to be chosen later. As for now, we set them in such a way that
    \begin{equation*}
        |R|=4XY=\kappa N,
    \end{equation*}
    where $\kappa$ is a positive value to be chosen later. Also, we assume that $X$ is reasonably bigger than $Y$. Now, we define the function $\Phi\colon\mathbb{Z}^2\to\mathbb{R^+}$ as
	\begin{equation*}
	\Phi(\mathbf{m})=\int_I\mathds{1}_{[\theta] R}(\mathbf{m})\de \theta
	\end{equation*}
	and then aim to find a parameter $Z=Z(N)$ such that for every $\mathbf{m}\in \mathbb{Z}^2$ it holds
	\begin{equation*}
	Z\Phi(\mathbf{m})\leq \begin{cases}
		|\mathbf{m}|^{-3} &\text{if } \arg(\mathbf{m})\in I\cup(I+\pi)\\
		|\mathbf{m}|^{-3-h} &\text{else}
	\end{cases}.
	\end{equation*}
    First, we consider all $\mathbf{m}\in \mathbb{Z}^2$ such that $|\mathbf{m}|\geq Y$. By some basic geometry, we find that $\Phi(\mathbf{m})\leq2\alpha$ whereas $\alpha$ is such that $|\mathbf{m}|\sin\alpha=Y$, and therefore we obtain
    \begin{equation*}
        \Phi(\mathbf{m})\leq\pi\frac{Y}{|\mathbf{m}|}.
    \end{equation*}
     Also, for every $\mathbf{m}\in\mathbb{Z}^2$ such that $\arg(\mathbf{m})\not\in I\cup(I+\pi)$, it is reasonable to assume $\Phi(\mathbf{m})=0$. Recall that in the sector $\arg(\mathbf{m})\in I\cup(I+\pi)$ we aim for $Z\Phi(\mathbf{m})\leq|\mathbf{m}|^{-3}$. Since for every $\mathbf{m}$ such that $|\mathbf{m}|\geq X$ it is reasonable to assume $\Phi(\mathbf{m})=0$, we are therefore led to choose
     \begin{equation*}
         Z\leq \frac{1}{\pi YX^2}.
     \end{equation*}
     On the other hand, we consider all $\mathbf{m}\in \mathbb{Z}^2$ such that $|\mathbf{m}|\leq Y$. It holds the trivial estimate $\Phi(\mathbf{m})\leq|I|\leq\pi$. It is enough to aim for $Z\Phi(\mathbf{m})\leq|\mathbf{m}|^{-3-h}$, and therefore we are led to choose
     \begin{equation*}
         Z\leq \frac{1}{\pi Y^{+3+h}}.
     \end{equation*}
     Thus, the choice
     \begin{equation*}
         Z\leq\min\left(\pi^{-1} Y^{-1}X^{-2}, \pi^{-1}Y^{-3-h} \right)
     \end{equation*}
    will suit us overall. By equalizing the two terms in the minimum, while keeping in mind the constrain $4XY=\kappa N$, we finally get
	\begin{equation*}
	X=c_1 N^{\frac{2+h}{4+h}},\quad Y=c_2 N^{\frac{2}{4+h}},\quad\text{and}\quad Z=c_3 N^{-\frac{6+2h}{4+h}},
	\end{equation*}
 whereas $c_i$ are positive values that eventually depend on $\kappa$ and $h$.\par
 Finally, for any set of $N$ points $\mathcal{P}_N=\{\mathbf{p}_j\}_{j=1}^N\subset\mathbb{T}^2$, by Parseval's identity and by Cassels-Montgomery lemma, we get
	\begin{align*}
		\int_\Xi\int_{\mathbb{T}^2}\left|\mathcal{D}(\mathcal{P}_N,\, [\boldsymbol{\tau},\xi] C)\right|^2\de \boldsymbol{\tau}\de\xi&=\sum_{\mathbf{m}\in\mathbb{Z}^2\setminus\{\mathbf{0}\}}\left|\sum_{j=1}^{N}e^{2\pi i \mathbf{m}\cdot \mathbf{p}_j}\right|^2\int_\Xi\left|\widehat{\mathds{1}}_{\xi C}(\mathbf{m})\right|^2\de \xi\\
		&\geq\tilde{c} \sum_{|\mathbf{m}|\geq\tilde{\rho}} \left|\sum_{j=1}^{N}e^{2\pi i \mathbf{m}\cdot \mathbf{p}_j}\right|^2Z\Phi(\mathbf{m})\\
		&=\tilde{c} Z\int_I\sum_{|\mathbf{m}|\geq\tilde{\rho}\,\text{ and }\,\mathbf{m}\in[\theta] R}\left|\sum_{j=1}^{N}e^{2\pi i \mathbf{m}\cdot \mathbf{p}_j}\right|^2\de \theta\\
		&\geq\tilde{c} Z |I| \left(\kappa N^2-c_{\tilde{\rho}} N^2\right),
	\end{align*}
	so that, by choosing $\kappa=2c_{\rho_0}$ in the last line, we obtain
	\begin{equation*}
	\int_\Xi\int_{\mathbb{T}^2}\left|\mathcal{D}(\mathcal{P}_N,\, [\boldsymbol{\tau},\xi] C)\right|^2\de \boldsymbol{\tau}\de\xi\geq c_4 N^{-\frac{6+2h}{4+h}}N^2=c_4 N^{\frac{2}{4+h}},
	\end{equation*}
 whereas $c_4$ is a positive value that eventually depends on $h$, $\tilde{\rho}$, $\tilde{c}$ and $|I|$.
\end{proof}

We now turn to the proof of our main results on the affine quadratic discrepancy of planar convex bodies.
\begin{proof}[Proof of Theorem~\ref{t1}]
	By Lemma~\ref{lem2}, we have that, uniformly for every $\omega\in\mathbb{T}_{2\pi}$, it holds
	\begin{equation}\label{Aux}
	\int_I\int_{0}^1\left|\widehat{\mathds{1}}_{[\delta, \theta] C}(\rho\,{\mathbf{u}(\omega)})\right|^2\de \delta\de \theta\asymp\rho^{-3}.
	\end{equation}
	Hence, by Theorem~\ref{t4} in the case of $h=0$, we get the lower bound
	\begin{equation*}
	\inf_{\# \mathcal{P}=N}\mathcal{D}_2(\mathcal{P},\, C,\,I)\succcurlyeq
	N^{1/2}.
	\end{equation*}\par
	In order to show the upper bound, we aim to find a suitable sampling for every $N$. To proceed, we first show it in the case of $N$ being a square, and then the general upper bound will follow from Lagrange's four-square theorem and the fact that, for $a_1,\ldots, a_4\geq0$, it holds
 \begin{equation*}
    \left(\sum_{j=1}^{4}a_j\right)^2\leq4\sum_{i=j}^{4}a_j^2.
 \end{equation*}
 Hence, let $N$ be a square and consider
	\begin{equation*}
	\mathcal{P}_N=\left\{\mathbf{p}_{h,j}\right\}_{h,j=1}^{N^{1/2}}=\left\{\left(\frac{h}{N^{1/2}},\frac{j}{N^{1/2}}\right)\right\}_{h,j=1}^{N^{1/2}}\subset\mathbb{T}^2.
	\end{equation*}
	By Parseval's identity, we get
	\begin{equation*}
		\mathcal{D}_2(\mathcal{P}_N,\, C,\,I)=\sum_{\mathbf{m}\in\mathbb{Z}^2\setminus\{\mathbf{0}\}}\left|\sum_{h=1}^{N^{1/2}}\sum_{j=1}^{N^{1/2}}e^{2\pi i \mathbf{m}\cdot \mathbf{p}_{h,j}}\right|^2\int_I\int_{0}^{1}\left|\widehat{\mathds{1}}_{[\delta, \theta] C}(\mathbf{m})\right|^2\de \delta\de \theta,
	\end{equation*}
	and in particular, notice that
	\begin{equation*}
	\sum_{h=1}^{N^{1/2}}\sum_{j=1}^{N^{1/2}}e^{2\pi i \mathbf{m}\cdot \mathbf{p}_{h,j}}=\begin{cases}
		N & \text{if } m_1=n_1N^{1/2}\text{ and } m_2=n_2N^{1/2}\text{ for some }\mathbf{n}\in\mathbb{Z}^2\\
		0 &\text{else}
	\end{cases}.
	\end{equation*}
	Finally, by \eqref{Aux}, we get
	\begin{equation*}
	\begin{split}
		\mathcal{D}_2(\mathcal{P}_N,\, C,\,I)&=\sum_{\mathbf{n}\in\mathbb{Z}^2\setminus\{\mathbf{0}\}}N^2\int_I\int_{0}^{1}\left|\widehat{\mathds{1}}_{[\delta, \theta] C}(\mathbf{n}N^{1/2})\right|^2\de \delta\de \theta\\
		&\preccurlyeq N^2\sum_{\mathbf{n}\in\mathbb{Z}^2\setminus\{\mathbf{0}\}} |\mathbf{n}|^{-3}N^{-3/2}\preccurlyeq N^{1/2}.
	\end{split}
	\end{equation*}
\end{proof}

	The proof of Theorem~\ref{t6} requires more attention. The first step to prove both lower and upper bound will be to individuate two sectors of $\mathbb{T}_{2\pi}$ where the averaged Fourier transform of $C$ has different magnitudes of decay.
\begin{proof}[Proof of Theorem~\ref{t6}]
	First, we prove the lower bound. By Lemma~\ref{t3} and Corollary~\ref{c2}, and by accounting the fact that not all points on the boundary of a planar convex body can be angular points, it follows that there exists an interval $I_1\subset\mathbb{T}_{2\pi}$ such that, uniformly for every $\omega\in I_1\cup(I_1+\pi)$, it holds
	\begin{equation*}
	\int_I\int_{0}^1\left|\widehat{\mathds{1}}_{[\delta, \theta] C}(\rho\,{\mathbf{u}(\omega)})\right|^2\de \delta\de \theta\asymp\rho^{-3}.
	\end{equation*}
    Moreover, by the results in Lemma~\ref{t3} and Proposition~\ref{r0}, we obtain that, uniformly for every $\omega\in\left(I_1\cup(I_1+\pi)\right)^\mathsf{c}$, it holds
	\begin{equation*}
	\int_I\int_{0}^1\left|\widehat{\mathds{1}}_{[\delta,\theta] C}(\rho\,\mathbf{u}(\omega))\right|^2\de \delta\de \theta\succcurlyeq\rho^{-2}\int_{\omega+I}\gamma_C^2(\theta,\rho^{-1})\de \theta\succcurlyeq\rho^{-4}.
	\end{equation*}
	Therefore, by applying Theorem~\ref{t4} in the case of $h=1$, we get the lower bound
	\begin{equation*}
	\inf_{\# \mathcal{P}=N}\mathcal{D}_2(\mathcal{P},\, C,\,I)\succcurlyeq
	N^{2/5}.
	\end{equation*}\par 
	Now, we proceed to show the upper bound. Since $|I|<\psi_C$, and by Lemma~\ref{t3} and Remark~\ref{r1}, we get the existence of an open interval $I_2$ such that, uniformly for every $\omega\in I_2\cup(I_2+\pi)$, it holds
	\begin{equation}\label{e1}
	\int_I\int_{0}^1\left|\widehat{\mathds{1}}_{[\delta,\theta] C}(\rho\,\mathbf{u}(\omega))\right|^2\de \delta\de \theta\asymp\rho^{-2}\int_{\omega+I}\gamma_C^2(\theta,\rho^{-1})\de \theta\asymp\rho^{-4}.
	\end{equation}
	Hence, by the results in Lemma~\ref{t3} and Corollary~\ref{c2}, we have that, uniformly for every $\omega\in\left(I_2\cup(I_2+\pi)\right)^\mathsf{c}$, it holds
	\begin{equation}\label{e2}
	\int_I\int_{0}^1\left|\widehat{\mathds{1}}_{[\delta,\theta] C}(\rho\,\mathbf{u}(\omega))\right|^2\de \delta\de \theta\preccurlyeq\rho^{-2}\int_{\omega+I}\gamma_C^2(\theta,\rho^{-1})\de \theta\preccurlyeq\rho^{-3}.
	\end{equation}
	We proceed to show an explicit construction of suitable samplings. First, let us do it for a number $N$ of points such that
 \begin{equation*}
     N=\lfloor n^{3/5}\rfloor\,\lfloor n^{2/5}\rfloor\quad\text{for some}\quad n\in\mathbb{N}.
 \end{equation*}
 Hence, set\begin{equation*}
G=\lfloor n^{3/5}\rfloor,\quad L=\lfloor n^{2/5}\rfloor,\quad J_G=[0,G-1]\cap\mathbb{N}\quad\text{and}\quad J_L=[0,L-1]\cap\mathbb{N}.
 \end{equation*}
 Now, take
 \begin{equation*}
     \frac{q_1}{q_2}\in\mathbb{Q}\quad\text{such that}\quad\gcd(q_1,q_2)=1\quad\text{and}\quad\arctan \frac{q_1}{q_2}\in I_2\cup(I_2+\pi),
 \end{equation*}
 so that the line $q_2y=q_1x$ makes an angle in $I_2\cup(I_2+\pi)$ with the $x$-axis. For the sake of simplicity, we set
 \begin{equation*}
     \tilde{\omega}=\arctan \frac{q_1}{q_2}.
 \end{equation*}
 To glimpse the idea behind the coming construction, notice that
	\begin{equation*}
	\sigma_{\tilde{\omega}}(x,y)=\left(q_1^2+q_2^2\right)^{-1/2}(q_2x-q_1y,\,q_1x+q_2y).
	\end{equation*}
	Hence, consider the set of points $\mathcal{P}_N\subset\mathbb{T}^2$ defined by
	\begin{equation*}
	\mathcal{P}_N=\left\{\mathbf{p}_j\right\}_{j=1}^N=\{\mathbf{p}_{\ell,g}\}_{\ell\in J_L,\,g\in J_G}\quad\text{with}\quad \mathbf{p}_{\ell,g}=\left(q_2\frac{\ell}{L}-q_1\frac{g}{G},\,q_1\frac{\ell}{L}+q_2\frac{g}{G}\right),
	\end{equation*}
	where the coordinates of $\mathbf{p}_{\ell,g}$ are to be intended modulo $1$ (in particular, repetition of points in $\mathcal{P}_N$ is admitted, and we refer to Figure~\ref{FF4} for an example).\begin{figure}
        \centering
        \includegraphics[width=0.9\linewidth]{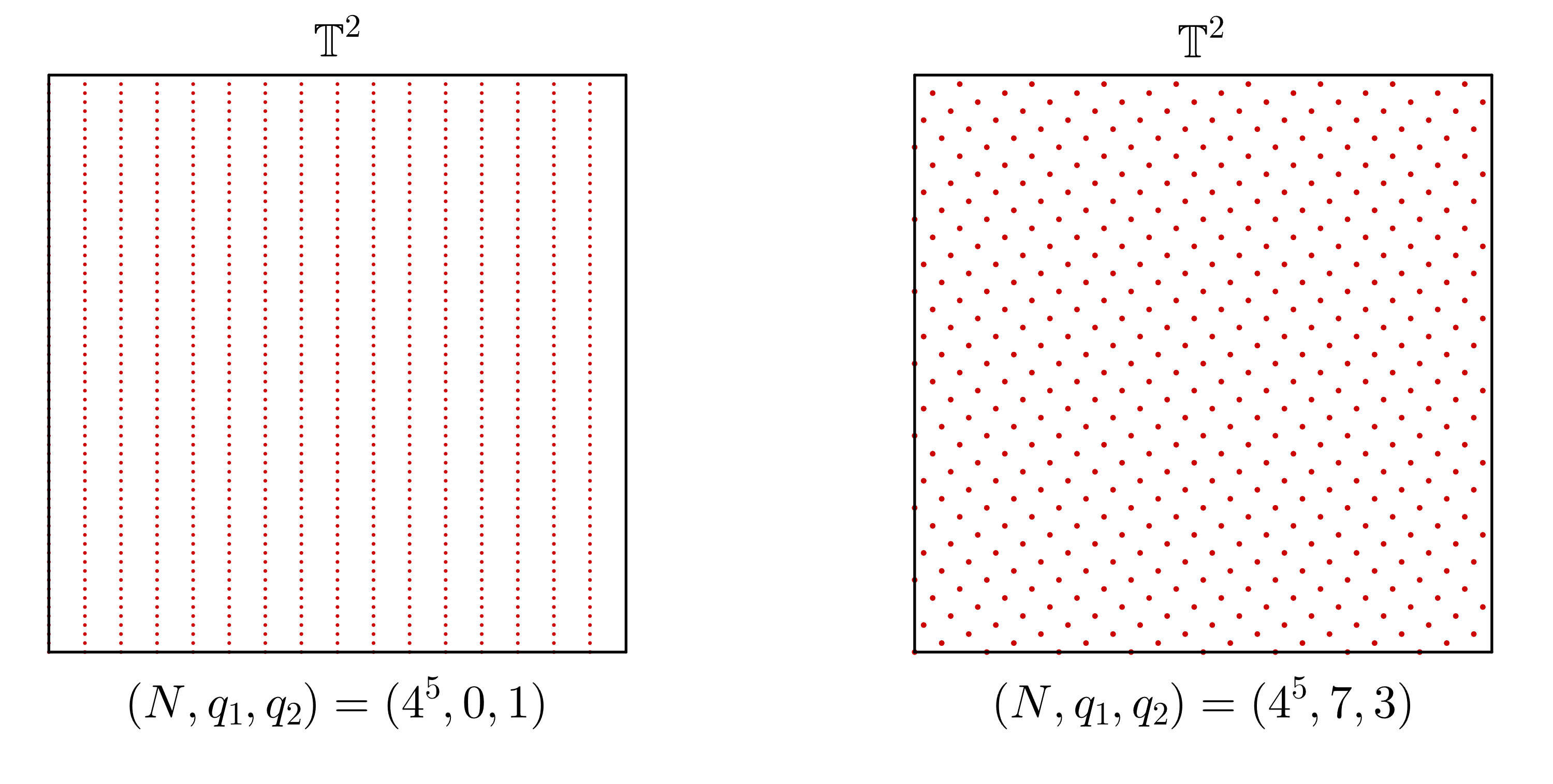}
        \caption{We depict two different dispositions of the points in $\mathcal{P}_N$ for $N=4^5$. The one on the left side corresponds to the starting lattice. The points on the right side are to be counted twice.}
        \label{FF4}
    \end{figure} Further, one may notice that $\mathbf{p}_{\ell,g}$ is the representative in $\mathbb{T}^2$ of the point $\left(\ell/L,g/G\right)\subset\mathbb{R}^2$ after a counterclockwise rotation by the angle $\tilde{\omega}$ and a dilation by the factor $(q_1^2+q_2^2)^{1/2}$. Again, by Parseval's identity, we obtain
	\begin{equation*}
	\int_{\mathbb{T}^2}\left|\mathcal{D}(\mathcal{P}_N, [\boldsymbol{\tau}, \delta, \theta] C)\right|^2\de \boldsymbol{\tau}=\sum_{\mathbf{m}\neq(0,0)}\left|\widehat{\mathds{1}}_{[\delta, \theta] C}(\mathbf{m})\right|^2\left|\sum_{g\in J_G}\sum_{\ell\in J_L}e^{2\pi i m\cdot \mathbf{p}_{\ell,g}}\right|^2.
	\end{equation*}
	Observe that 
	\begin{equation*}
	\begin{split}
		&\sum_{g\in J_G}\sum_{\ell\in J_L}e^{2\pi i\left(\frac{\ell}{L}(q_2m_1+q_1m_2)+\frac{g}{G}(q_2m_2-q_1m_1)\right)}=
		\\&=\begin{cases}
			GL &\text{if } q_2m_1+q_1m_2=n_1L\text{ and } q_2m_2-q_1m_1=n_2G\text{ for some }\mathbf{n}\in\mathbb{Z}^2\\
			0 &\text{else}
		\end{cases},
	\end{split}
	\end{equation*}
	hence we are looking for all non-zero $\mathbf{m}\in\mathbb{Z}^2$ for which there exist some $\mathbf{n}\in\mathbb{Z}^2$ such that
	\begin{equation*}
	\begin{cases}
		m_1=\frac{1}{q_1^2+q_2^2}(q_2n_1L-q_1n_2G)\\
		m_2=\frac{1}{q_1^2+q_2^2}(q_1n_1L+q_2n_2G)
	\end{cases}.
	\end{equation*}
	We label as $\mathcal{R}$ the set of all the $\mathbf{m}\in\mathbb{Z}^2$ that happen to be solutions to the latter system. Furthermore, we consider the auxiliary set
	\begin{equation*}
	\tilde{\mathcal{R}}=(q_1^2+q_2^2)^{-1/2}\left\{(n_1L,n_2G)\,:\,\mathbf{n}\in\mathbb{Z}^2\setminus\{\mathbf{0}\}\right\},
	\end{equation*}
	and in particular, we notice that $[-\tilde{\omega}]\mathcal{R}\subset\tilde{\mathcal{R}}$.

 Again, since the Fourier transform and rotations commute, we get
	\begin{equation}\label{e6}
	\begin{split}
		\sum_{\mathbf{m}\in\mathcal{R}}\left|\widehat{\mathds{1}}_{[\delta, \theta] C}(\mathbf{m})\right|^2&=\sum_{\mathbf{m}\in[-\tilde{\omega}]\mathcal{R}}\left|\widehat{\mathds{1}}_{[\delta, \theta] C}(\sigma_{\tilde{\omega}}\mathbf{m})\right|^2\\
		&=\sum_{\mathbf{m}\in[-\tilde{\omega}]\mathcal{R}}\left|\widehat{\mathds{1}}_{[\delta, \theta-\tilde{\omega}] C}(\mathbf{m})\right|^2\leq\sum_{\mathbf{m}\in\tilde{\mathcal{R}}}\left|\widehat{\mathds{1}}_{[\delta, \theta-\tilde{\omega}] C}(\mathbf{m})\right|^2.
	\end{split}
	\end{equation}
    In order to estimate the latter quantity, we distinguish between two different regions of $\tilde{\mathcal{R}}$. First, we let
    \begin{equation*}
         \alpha\in\left(0,\frac{\pi}{2}\right)\quad\text{be such that}\quad[-\alpha,\alpha]\subset \left(I_2\cup(I_2+\pi)-\tilde{\omega}\right),
    \end{equation*}
    and then we split $\tilde{\mathcal{R}}$ in the region
	\begin{equation*}
	V=\left\{\mathbf{m}\in\tilde{\mathcal{R}}\,\colon\,\arg \mathbf{m}\in [-\alpha,\alpha]\right\}
	\end{equation*}
	and its complementary $V^\mathsf{c}$. In particular, the condition $\arg \mathbf{m}\in[-\alpha,\alpha]$ in the definition of $V$ translates into the requirement
 \begin{equation*}
     |n_2| G\leq |n_1| L \tan \alpha.
 \end{equation*}
 Hence, by \eqref{e6}, we have 
	\begin{align*}
		&\mathcal{D}_2(\mathcal{P}_N,\, C,\,I)\leq\\
        &\leq\int_I\int_{0}^{1}G^2L^2\sum_{\mathbf{m}\in\tilde{\mathcal{R}}}\left|\widehat{\mathds{1}}_{[\delta, \theta-\tilde{\omega}] C}(\mathbf{m})\right|^2\de \delta\de \theta\\
		&= G^2L^2\sum_{\mathbf{m}\in V}\int_I\int_{0}^{1}\left|\widehat{\mathds{1}}_{[\delta, \theta-\tilde{\omega}] C}(\mathbf{m})\right|^2\de\delta\de \theta +\\
        &+G^2L^2\sum_{\mathbf{m}\in V^\mathsf{c}}\int_I\int_{0}^{1}\left|\widehat{\mathds{1}}_{[\delta, \theta-\tilde{\omega}] C}(\mathbf{m})\right|^2\de\delta\de \theta.
	\end{align*}
	For the first sum in the last term, by \eqref{e1}, we get
	\begin{align*}
		&G^2L^2\sum_{\mathbf{m}\in V}\int_I\int_{0}^{1}\left|\widehat{\mathds{1}}_{[\delta, \theta-\tilde{\omega}] C}(\mathbf{m})\right|^2\de\delta\de \theta\preccurlyeq\\
        &\preccurlyeq G^2L^2\sum_{\mathbf{m}\in V}|\mathbf{m}|^{-4}\\
		&\preccurlyeq G^2L^2  (q_1^2+q_2^2)^{2}(1+\tan \alpha)^{-4}\sum_{n_1=1}^{+\infty}\:\sum_{n_2=0} ^{n_1 \frac{L}{G} \tan \alpha}\left(n_1L\right)^{-4}\\
		&\preccurlyeq  G^{2}L^{-2}\sum_{n_1=1}^{+\infty}n_1^{-4}\left(n_1\frac{L}{G}\tan\alpha+1\right)\preccurlyeq G^2L^{-2} \preccurlyeq N^{2/5}.
	\end{align*}
	On the other hand, for the second sum, we get
	\begin{align*}
		&G^2L^2\sum_{\mathbf{m}\in V^\mathsf{c}}\int_I\int_{0}^{1}\left|\widehat{\mathds{1}}_{[\delta, \theta-\tilde{\omega}] C}(\mathbf{m})\right|^2\de\delta\de \theta\preccurlyeq\\
        &\preccurlyeq  G^2L^2\sum_{\mathbf{m}\in V^\mathsf{c}}|\mathbf{m}|^{-3}\\
		&\preccurlyeq G^2L^2  (q_1^2+q_2^2)^{3/2}(1+\cot \alpha)^{-3}\sum_{n_2=1}^{+\infty}\:\sum_{n_1=0}^{n_2 \frac{G}{L} \cot \alpha}\left(n_2G\right)^{-3}\\
		&\preccurlyeq G^{-1}L^2\sum_{n_2=1}^{+\infty}n_2^{-3}\left(n_2\frac{G}{L}\cot\alpha+1\right)\preccurlyeq L\preccurlyeq N^{2/5},
	\end{align*}
	and we can conclude that the initial claim holds for all $N$ of the form $N=\lfloor n^{3/5}\rfloor\,\lfloor n^{2/5}\rfloor$.\par 
    In order to prove that there is a suitable choice of points for every positive integer $N$, consider the following recursive definition
	\begin{equation*}
	n_j=\max\left\{n\in\mathbb{N}\;\colon\lfloor n^{3/5}\rfloor\,\lfloor n^{2/5}\rfloor \leq N-\sum_{i=1}^{j-1}\,\lfloor n_i^{3/5}\rfloor\,\lfloor n_i^{2/5}\rfloor\right\}\quad\text{for}\quad j\in\mathbb{N}\setminus\{0\},
	\end{equation*}
	whereas improper sums are conventionally considered as zeros.
	Now, notice that the latter definition implies that
	\begin{equation*}
 \begin{split}
	N-\sum_{i=1}^{j-1}\,\lfloor n_i^{3/5}\rfloor\,\lfloor n_i^{2/5}\rfloor&\leq\lfloor (n_j+1)^{3/5}\rfloor\,\lfloor (n_j+1)^{2/5}\rfloor\\
    &\leq\lfloor n_j^{3/5}\rfloor\,\lfloor n_j^{2/5}\rfloor+\lfloor n_j^{3/5}\rfloor+\lfloor n_j^{2/5}\rfloor+1,
  \end{split}
	\end{equation*}
	and therefore, it follows that
	\begin{equation*}
	N-\sum_{i=1}^{j}\,\lfloor n_i^{3/5}\rfloor\,\lfloor n_i^{2/5}\rfloor\leq2n_j^{3/5}.
	\end{equation*}
	Again, by the definition of $n_j$, it is easy to see that
	\begin{equation*}
	\frac{n_1}{2}\leq N\quad\text{and}\quad
	\frac{n_{j+1}}{2}\leq N-\sum_{i=1}^{j}\,\lfloor n_i^{3/5}\rfloor\,\lfloor n_i^{2/5}\rfloor,
	\end{equation*}
	so that by induction, we get
	\begin{equation*}
	N-\sum_{i=1}^{j}\,\lfloor n_i^{3/5}\rfloor\,\lfloor n_i^{2/5}\rfloor\leq2n_j^{3/5}\leq2^{4}n_{j-1}^{9/25}\leq2^{2j}N^{\left(3/5\right)^j}.
	\end{equation*}
	In particular, notice that
	\begin{equation*}
	N-\sum_{j=1}^{4}\lfloor n_j^{3/5}\rfloor\,\lfloor n_j^{2/5}\rfloor\leq2^8 N^{\left(3/5\right)^{4}}=o(N^{1/5}).
	\end{equation*}
	Finally, we associate a choice of points as in the previous construction to every $N_j=\lfloor n_j^{3/5}\rfloor\,\lfloor n_j^{2/5}\rfloor$ with $1\leq j\leq4$, and we do not get bothered by the remaining points since the reminder is $o(N^{1/5})$. The conclusion follows at once since, for $a_1,\ldots,a_5\geq0$, it holds
 \begin{equation*}
	\left(\sum_{j=1}^{5}a_j\right)^2\leq5\sum_{j=1}^{5}a_j^2.
	\end{equation*}
\end{proof}

\section{Intermediate Orders of Discrepancy}\label{S4}

We now prove that, for an interval of angles
\begin{equation}\label{NotazInter}
    I(\phi)=\left[-\frac{\phi}{2},\frac{\phi}{2}\right]\subset\mathbb{T}_{2\pi}\quad\text{with}\quad\phi\in(0,\pi),\quad \text{and for}\quad \alpha\in(1,+\infty),
\end{equation}
there exists a planar convex body $C(\phi,\alpha)$ with piecewise-$\mathcal{C}^\infty$ boundary such that it holds
\begin{equation*}
\inf_{\# \mathcal{P}=N}\mathcal{D}_2(\mathcal{P},\, C(\phi,\alpha),\,I(\phi))\asymp N^{\frac{2\alpha}{4\alpha+1}}.
\end{equation*}
For the sake of notation, the letter $\varepsilon$ will stand for a generic positive small value throughout this section. Moreover, for an interval $U\subseteq [0,+\infty)$ and two positive functions $f$ and $g$ defined on $U$, we say that for $x\in U$ it holds
\begin{equation*}
    f(x)\approx g(x)
\end{equation*}
to intend that there exist positive values $c_1$ and $c_2$ (which eventually depend on $\alpha$ and $\phi$) such that, for every $x\in U$, it holds
\begin{equation*}
c_1\, g(x)\leq f(x)\leq c_2\, g(x).
\end{equation*}\par
The key to obtaining these intermediate orders is to build such a convex body in a way that $\psi_{C(\phi,\alpha)}=\phi$. For the sake of construction, first, consider a planar convex body $H(\phi,\alpha)$ such that it has a centre of symmetry and such that it is symmetric with respect to the line
\begin{equation*}
    y=x\tan\!\left(\frac{\pi}{2}-\frac{\phi}{2}\right).
\end{equation*}
Moreover, build it in such a way that 
\begin{equation*}
\left\{(x,x^\alpha)\in\mathbb{R}^2\,\colon\, x\in[0,\varepsilon]\right\}\subset\partial H(\phi,\alpha).
\end{equation*}
Last, construct $H(\phi,\alpha)$ in such a way that its boundary is $\mathcal{C}^\infty$ except at the origin and at its symmetric counterpart (see Figure~\ref{FF5}).\begin{figure}
        \centering
        \includegraphics[width=0.9\linewidth]{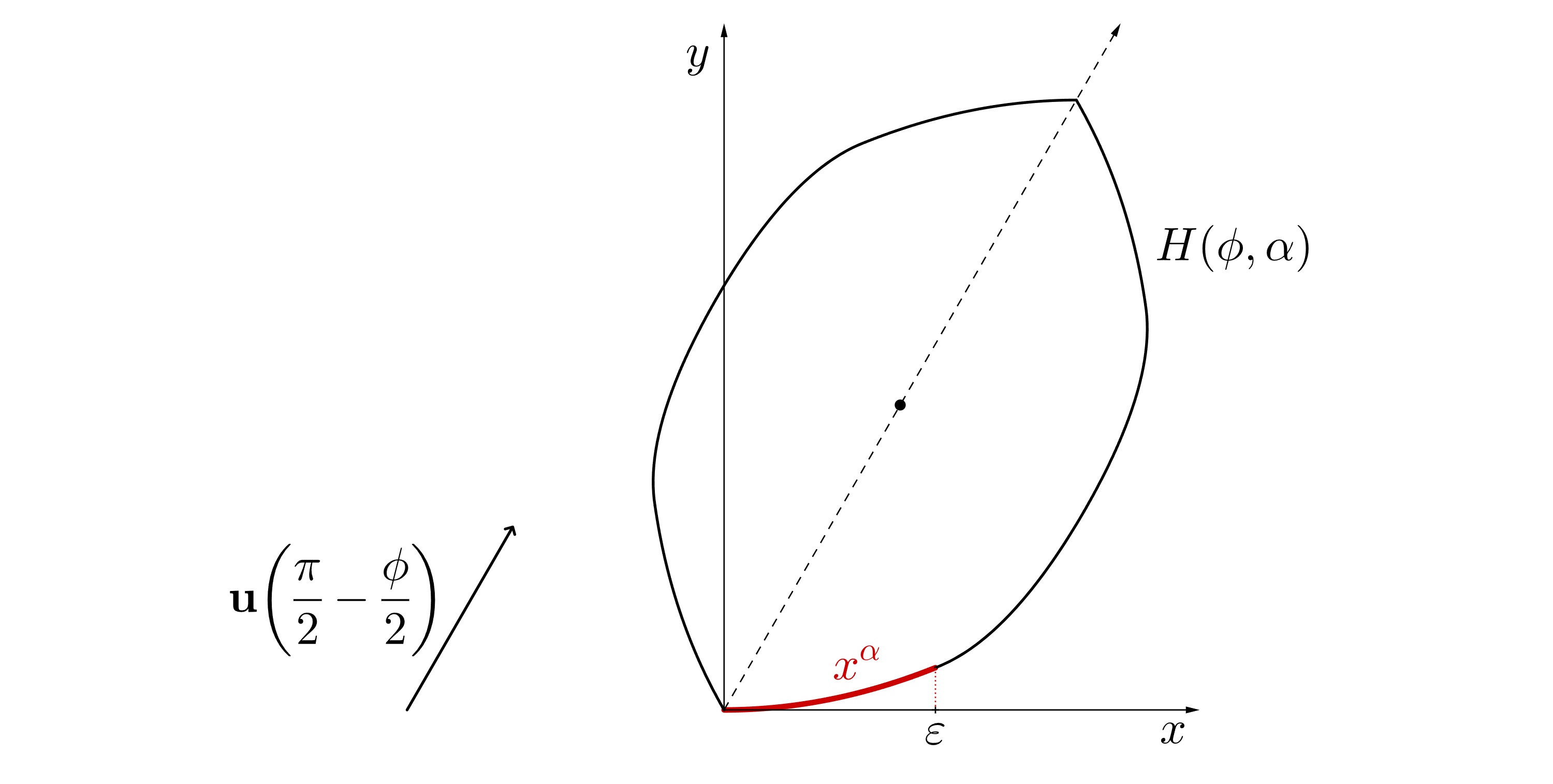}
        \caption{A depiction of $H(\phi,\alpha)$.}
        \label{FF5}
    \end{figure} Hence, in order to evaluate its affine quadratic discrepancy, it is sufficient to get estimates for the chords of $H(\phi,\alpha)$ about the origin. By symmetry, we can restrict ourselves to study the directions
\begin{equation*}
    \uthe\quad\text{for}\quad\theta\in\left[\frac{\pi}{2}-\frac{\phi}{2},\frac{\pi}{2}+\varepsilon\right].
\end{equation*}
First, we present an auxiliary technical result.
\begin{lem}\label{r2}

Let $\alpha$ and $\beta$ be positive numbers, and let $g\colon\mathbb{R}^+\to\mathbb{R}^+$ be such that
\begin{equation*}
g(x)\approx\begin{cases}
	x^\alpha& \textnormal{if}\quad 0\leq x< 1\\
	x^\beta& \textnormal{if}\quad x\geq 1 
\end{cases}.
\end{equation*}
If $x_y$ is such that $g(x_y)=y$, then it holds
\begin{equation*}
x_y\approx\begin{cases}
	y^{1/\alpha} & \textnormal{if}\quad 0\leq x<1 \\
	y^{1/\beta} & \textnormal{if}\quad y\geq1
\end{cases}.
\end{equation*}
\end{lem}
\begin{proof}
	By hypothesis, there exist two positive values $c_1$ and $c_2$ such that it holds
	\begin{equation*}
	\begin{cases}
		c_1\,x^\alpha\leq g(x)\leq c_2\,x^\alpha& \text{if}\quad 0\leq x< 1\\
		c_1\,x^\beta\leq g(x)\leq c_2\,x^\beta& \text{if}\quad x\geq 1 
	\end{cases}.
	\end{equation*}
	If $y<c_1$ then we necessarily have $0\leq x_y\leq 1$, and therefore
    \begin{equation*}
	    c_1\,x_y^\alpha\leq y\leq c_2\,x_y^\alpha.
    \end{equation*}
 Rearranging, one gets that
	\begin{equation*}
	c_1^{1/{\alpha}}\,x_y\leq y^{1/{\alpha}}\leq c_2^{1/{\alpha}}\,x_y\quad\text{for}\quad y\in[0,c_1).
	\end{equation*}
	On the other hand, if $y>c_2$ then we necessarily have $x_y\geq 1$, and therefore
 \begin{equation*}
     c_1\,x_y^\beta\leq y\leq c_2\,x_y^\beta.
 \end{equation*}
 Rearranging, one gets that 
	\begin{equation*}
	c_1^{1/{\beta}}\,x_y\leq y^{1/{\beta}}\leq c_2^{1/{\beta}}\,x_y\quad\text{for}\quad y\in(c_2,+\infty).
	\end{equation*}
	The claim follows since, for every $y\in[c_1,c_2]$, we have that $x_y$ is bounded away from $0$ or $+\infty$.
\end{proof}

Let us first study the case when $s_{H(\phi,\alpha)}^o(\theta)$ is the origin, or in other words, when $\theta\in\left[\frac{\pi}{2}-\frac{\phi}{2},\frac{\pi}{2}\right]$.

\begin{prop}\label{l8}
	Let $H(\phi,\alpha)$ be as previously defined. Uniformly for every $\theta\in\left[\frac{\pi}{2}-\frac{\phi}{2},\frac{\pi}{2}\right]$, it holds
	\begin{equation*}
	\left|K_{H(\phi,\alpha)}(\theta,\rho^{-1})\right|\asymp\begin{cases}
		\rho^{-1/\alpha} & \textnormal{if}\quad 0\leq\frac{\pi}{2}-\theta<\rho^{\frac{1-\alpha}{\alpha}}\\
		\rho^{-1} (\frac{\pi}{2}-\theta)^{-1} & \textnormal{if}\quad \rho^{\frac{1-\alpha}{\alpha}}\leq\frac{\pi}{2}-\theta\leq\frac{\phi}{2}
	\end{cases}.
	\end{equation*}
\end{prop}
\begin{proof}
	By symmetry, there exists $\rho_0>0$ such that, for every $\theta\in\left[\frac{\pi}{2}-\frac{\phi}{2},\frac{\pi}{2}\right]$ and for every $\rho\geq\rho_0$, we have that the part of the chord $K_{H(\phi,\alpha)}(\theta,\rho^{-1})$ at the right of $y=x\tan(\frac{\pi}{2}-\frac{\phi}{2})$ is longer than the part at the left. Hence, by considering the auxiliary shape
	\begin{equation*}
	F(\alpha)=\left\{(x,y)\in\mathbb{R}^2\,\colon\, x\geq0\,\text{ and }\,y\geq x^\alpha\right\},
	\end{equation*}
	it is not difficult to see that, uniformly for every $\theta\in\left[\frac{\pi}{2}-\frac{\phi}{2},\frac{\pi}{2}\right]$, it holds
 \begin{equation*}
     \left|K_{H(\phi,\alpha)}(\theta,\rho^{-1})\right|\asymp \left|K_{F(\alpha)}(\theta,\rho^{-1})\right|.
 \end{equation*}
 Therefore, we can restrict ourselves to studying the chords of $F(\alpha)$. Now, for the sake of notation, we let
 \begin{equation*}
     x_+=x_{F(\alpha)}^+(\theta,\rho^{-1})\quad\text{be the abscissa of}\quad s_{F(\alpha)}^+(\theta,\rho^{-1}),
 \end{equation*}
 and define $x_-$ analogously. It is immediate to see that, for every $\theta\in\left[\frac{\pi}{2}-\frac{\phi}{2},\frac{\pi}{2}\right]$, we have $x_-=0$, and it also holds
	\begin{equation*}
	\left|K_{F(\alpha)}(\theta,\rho^{-1})\right|=\frac{x_+-x_-}{\sin\theta}.
	\end{equation*}
	On the other hand, $x_+$ is the abscissa of the intersection in $x\geq 0$ between the curve $y=x^\alpha$ and the straight line
	\begin{equation*}y - \rho^{-1}\,\sin\theta =-\frac{1}{\tan\theta} (x-\rho^{-1}\cos\theta).
	\end{equation*} Rearranging, we have that $x_+$ is a solution of 
	\begin{equation*}
	x(x^{\alpha-1}\,\sin\theta+\cos\theta) = \rho^{-1},
	\end{equation*}
	and by the normalization
 \begin{equation*}
      z=x^{\alpha-1}\,\tan\theta,
 \end{equation*}
 we get the equation
	\begin{equation*}
	f(z)=z^{\frac{1}{\alpha-1}}\left(z+1\right)=\frac{(\tan\theta)^{\frac{\alpha}{\alpha-1}}}{\rho\sin\theta}.
	\end{equation*}

	Notice that it holds
	\begin{equation*}
	f(z)\approx\begin{cases}
		z^{\frac{1}{\alpha-1}} & \text{if}\quad 0\leq z < 1\\
		z^{\frac{\alpha}{\alpha-1}} & \text{if}\quad  z\geq 1
	\end{cases},
	\end{equation*}
	and by applying Lemma~\ref{r2}, and the fact that for $\theta\in\left[\frac{\pi}{2}-\frac{\phi}{2},\frac{\pi}{2}\right]$ it holds 
 \begin{equation*}
     \sin\theta\approx1\quad\text{and}\quad\cot\theta\approx \frac{\pi}{2}-\theta, 
 \end{equation*}
 it follows that
	\begin{equation*}
	x_+^{\alpha-1}\left(\frac{\pi}{2}-\theta\right)^{-1}\approx\begin{cases}
		 \rho^{\frac{1-\alpha}{\alpha}}\left(\frac{\pi}{2}-\theta\right)^{-1} & \text{if}\quad0\leq \frac{\pi}{2}-\theta < \rho^{\frac{1-\alpha}{\alpha}}\\
		 \rho^{1-\alpha}\left(\frac{\pi}{2}-\theta\right)^{-\alpha}& \text{if}\quad \rho^{\frac{1-\alpha}{\alpha}}\leq \frac{\pi}{2}-\theta \leq\frac{\phi}{2}
	\end{cases}.
	\end{equation*}
	By a last rearrangement, we get
	\begin{equation*}
	x_+\approx\begin{cases}
		\rho^{-1/\alpha} & \text{if}\quad0\leq \frac{\pi}{2}-\theta < \rho^{\frac{1-\alpha}{\alpha}}\\
		\rho^{-1}\left(\frac{\pi}{2}-\theta\right)^{-1}& \text{if}\quad \rho^{\frac{1-\alpha}{\alpha}}\leq \frac{\pi}{2}-\theta \leq\frac{\phi}{2}
	\end{cases}.
	\end{equation*}
\end{proof}

We now turn to estimating $\left|K_{H(\phi,\alpha)}(\theta,\rho^{-1})\right|$ in the case of $\theta\in\left[\frac{\pi}{2},\frac{\pi}{2}+\varepsilon\right]$. Again, we make use of an auxiliary shape. Namely, consider
\begin{equation*}
G(\alpha)=\left\{(x,y)\in\mathbb{R}^2\,\colon\, y\geq|x|^\alpha\right\},
\end{equation*}
and as before, notice that, uniformly for every $\theta\in\left[\frac{\pi}{2},\frac{\pi}{2}+\varepsilon\right]$, it holds
\begin{equation*}
\left|K_{H(\phi,\alpha)}(\theta,\rho^{-1})\right|\asymp \left|K_{G(\alpha)}(\theta,\rho^{-1})\right|.
\end{equation*}
First, we need a technical observation on the chords of $G(\alpha)$.
\begin{prop}\label{Ausilio}
	Let $G(\alpha)$ be as previously defined. There exists a positive value $c_\alpha$ such that, for every $\theta\in\left[\frac{\pi}{2},\frac{\pi}{2}+\varepsilon\right]$ and for every $\rho\geq1$, it holds
	\begin{equation*}
	\left|K_{G(\alpha)}^-(\theta,\rho^{-1})\right|\leq c_\alpha\left|K_{G(\alpha)}^+(\theta,\rho^{-1})\right|.
	\end{equation*}
\end{prop}
\begin{proof}
	For the sake of notation, let
 \begin{equation*}
     x_o=x_{G(\alpha)}^o(\theta)\quad\text{be the abscissa of}\quad s_{G(\alpha)}^o(\theta).
 \end{equation*}
 Moreover, we let
 \begin{equation*}
 x_+=x_{G(\alpha)}^+(\theta,\rho^{-1})\quad\text{be the abscissa of}\quad s_{G(\alpha)}^+(\theta,\rho^{-1}),    
 \end{equation*}
 and define $x_-$ analogously. With the help of Figure~\ref{FF6}, notice that, for every $\theta\in\left[\frac{\pi}{2},\frac{\pi}{2}+\varepsilon\right]$, it holds
	\begin{equation*}
	\left|K_{G(\alpha)}^-(\theta,\rho^{-1})\right|\sin\theta \leq x_o-x_-\quad\text{and}\quad \left|K_{G(\alpha)}^+(\theta,\rho^{-1})\right|\sin\theta \geq x_+-x_o,
	\end{equation*}\begin{figure}
        \centering
        \includegraphics[width=0.9\linewidth]{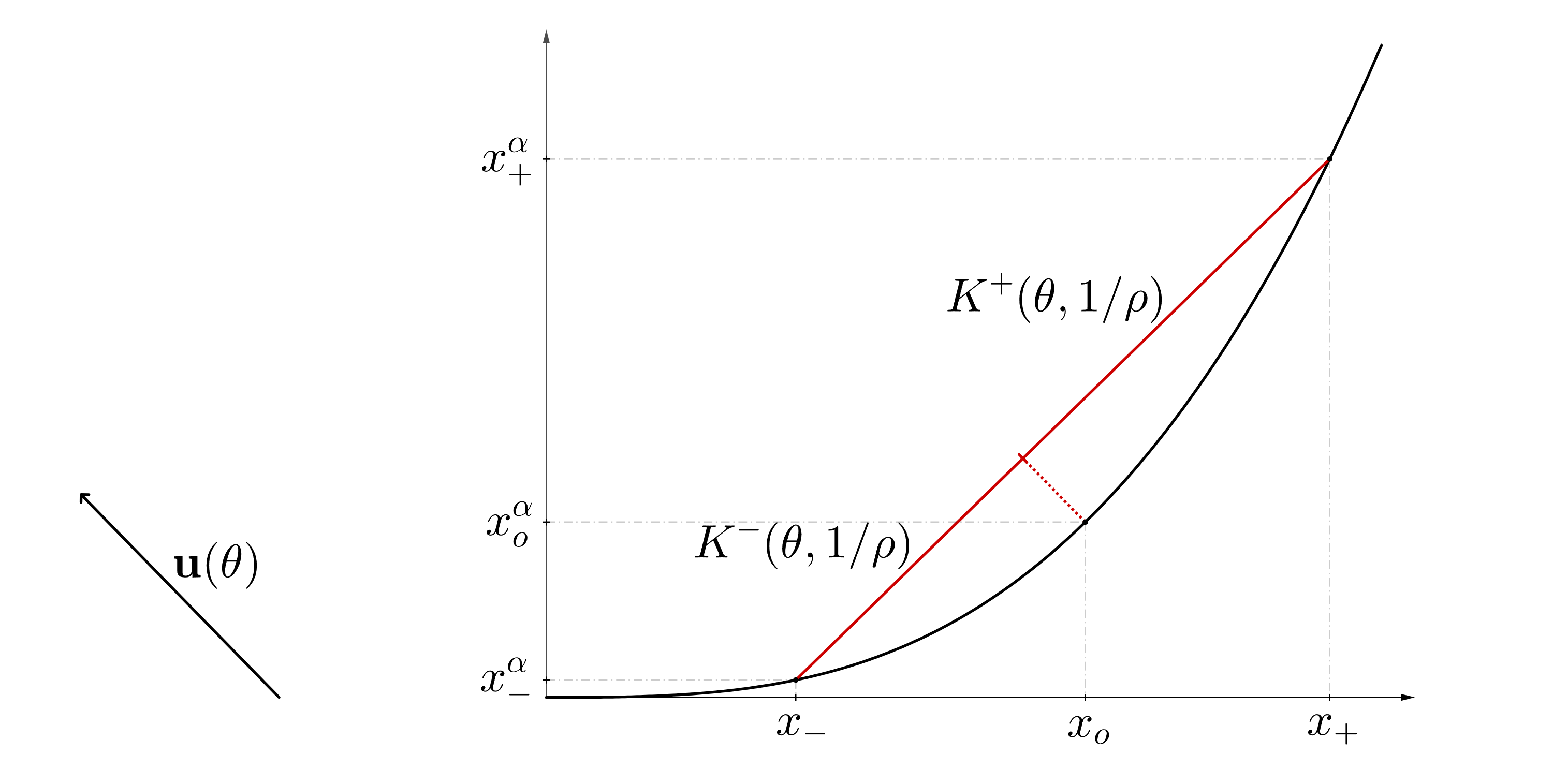}
        \caption{An auxiliary image for the proof of Proposition~\ref{Ausilio}. For simplicity, we omit to write $G(\alpha).$}
        \label{FF6}
    \end{figure}and therefore, it is enough to show that
 \begin{equation*}
     x_o-x_-\leq c_\alpha(x_+-x_o).
 \end{equation*}
 Indeed, for every $\theta\in\left[\frac{\pi}{2},\frac{\pi}{2}+\varepsilon\right]$, we have that $x_-$ and $x_+$ are the abscissas of the intersections of the curve $y=|x|^\alpha$ with the straight line
	\begin{equation*}y=(x-x_o)\alpha x_o^{\alpha-1}+x_o^\alpha+\frac{1}{\rho\sin\theta}.
	\end{equation*}
	Equalizing, and with the normalization $z=\frac{x-x_o}{x_o}$, we get to the equation \begin{equation}\label{e3}
		f(z)=|z+1|^{\alpha}-z\alpha-1=\frac{1}{x_o^\alpha\rho\sin\theta},
	\end{equation}
	and we also remark that, for every $\theta\in\left[\frac{\pi}{2},\frac{\pi}{2}+\varepsilon\right]$, both $x_o$ and $\sin\theta$ are non-negative. Hence, the conclusion follows once we show that
 \begin{equation*}
     f(z)\leq f(-c_\alpha\,z)\quad\text{for every}\quad z\geq 0,
 \end{equation*}
 since this would imply
	\begin{equation*}
	\frac{x_+-x_o}{x_o}\geq-\frac{1}{c_\alpha}\cdot\frac{x_--x_o}{x_o}.
	\end{equation*}
	Last, it is not difficult to see that by choosing $c_\alpha=2^\alpha$ then, for every $z\geq 0$, it holds
	\begin{equation*}
	f'(z)=\alpha\left(|z+1|^{\alpha-1}-1\right) \leq\alpha 2^\alpha \left(|2^\alpha z-1|^{\alpha-1}+1\right)=\frac{\partial}{\partial z}f(-2^\alpha z),
	\end{equation*}
	and indeed, one has
 \begin{equation*}
     z+1\leq2\quad\text{for}\quad 0\leq z <1,\quad\text{and}\quad z+1\leq2(2z-1)\quad\text{for}\quad z\geq1.
 \end{equation*}
\end{proof}

Now, we proceed to estimate the chords in the case of $\theta\in\left[\frac{\pi}{2},\frac{\pi}{2}+\varepsilon\right]$.

\begin{prop}\label{l9}
	Let $H(\phi,\alpha)$ be as previously defined. Uniformly for every $\theta\in\left[\frac{\pi}{2},\frac{\pi}{2}+\varepsilon\right]$, it holds
	\begin{equation*}
	\left|K_{H(\phi,\alpha)}(\theta,\rho^{-1})\right|\asymp\begin{cases}
		\rho^{-1/\alpha} & \textnormal{if}\quad0\leq\theta-\frac{\pi}{2}<\rho^{\frac{1-\alpha}{\alpha}}\\
		\rho^{-1/2}\left(\theta-\frac{\pi}{2}\right)^{\frac{2-\alpha}{2(\alpha-1)}}& \textnormal{if}\quad \rho^{\frac{1-\alpha}{\alpha}}\leq \theta-\frac{\pi}{2} \leq \varepsilon
	\end{cases}.
	\end{equation*}
\end{prop}
\begin{proof}
	We have already noted that we can equivalently study the chords of the auxiliary shape $G(\alpha)$, and therefore, we define $x_-$, $x_o$, $x_+$, and $f$, as in the Proposition~\ref{Ausilio}. Since 
	\begin{equation*}
	\left|K_{G(\alpha)}(\theta,\rho^{-1})\right|\sin\theta=(x_+-x_-),
	\end{equation*}
	then, by the previous lemma, it is enough to estimate $(x_+-x_o)$. As before, $x_+$ is a solution of
	\begin{equation*}|x|^\alpha=(x-x_o)\alpha x_o^{\alpha-1}+x_o^\alpha+\frac{1}{\rho\sin\theta},
	\end{equation*}
	and again by the normalization $z=\frac{x-x_o}{x_o}$, we get \eqref{e3}. In particular, we remark that the solution $x_+$ corresponds to the range $z\geq0$. Now, by applying Taylor's formula with integral reminder to $f$, we get
	\begin{equation*}
	f(z)=\alpha(\alpha-1)\int_{0}^{z}(1+t)^{\alpha-2}(z-t)\de t.
	\end{equation*}
	Notice that for $z\in[0,1)$ it holds
	\begin{equation*}
	\int_{0}^{z}(1+t)^{\alpha-2}(z-t)\de t\approx \int_{0}^{z}(z-t)\de t=\frac{z^2}{2}.
	\end{equation*}
	On the other hand, for $z\in[1,+\infty)$ it holds
	\begin{equation*}
	\begin{split}
		\int_{0}^{z}(1+t)^{\alpha-2}(z-t)\de t&=\int_{0}^{z/2}(1+t)^{\alpha-2}(z-t)\de t+\int_{\frac{z}{2}}^{z}(1+t)^{\alpha-2}(z-t)\de t\\
		&\approx z\int_{0}^{z/2}(1+t)^{\alpha-2}\de t+z^{\alpha-2}\int_{\frac{z}{2}}^{z}(z-t)\de t\\
		&=\frac{z}{\alpha-1}\left(\left(1+z/2\right)^{\alpha-1}-1\right)+z^{\alpha-2}\frac{z^2}{8}\approx z^{\alpha}.
	\end{split}
	\end{equation*}
	Hence, we get
	\begin{equation*}
	f(z)\approx\begin{cases}
		z^2&\text{if}\quad0\leq z<1\\
		z^\alpha&\text{if}\quad z\geq1
	\end{cases},
	\end{equation*}
	and if we consider \eqref{e3}, by applying Lemma~\ref{r2}, and by the fact that for $\theta\in\left[\frac{\pi}{2},\frac{\pi}{2}+\varepsilon\right]$ it holds $\sin\theta\approx1$, then it follows that
	\begin{equation}\label{e4}
	\frac{x_+-x_o}{x_o}\approx\begin{cases}
		\rho^{-1/2}x_o^{-\alpha/2}&\text{if} \quad 0\leq \rho^{-1}x_o^{-\alpha}<1\\
		\rho^{-1/\alpha}x_o^{-1}&\text{if}\quad \rho^{-1}x_o^{-\alpha}\geq1
	\end{cases}.
	\end{equation}
	Last, by the definition of $x_o$, we have 
	\begin{equation*}
		\alpha x_o^{\alpha-1}=\left.\frac{d}{dx}x^\alpha\right|_{x=x_o}=\tan\!\left(\theta-\frac{\pi}{2}\right),
	\end{equation*}
	and therefore, we get that for $\theta\in\left[\frac{\pi}{2},\frac{\pi}{2}+\varepsilon\right]$ it holds
 \begin{equation*}
     x_o\approx\left(\theta-\frac{\pi}{2}\right)^{\frac{1}{\alpha-1}}.
 \end{equation*}
 The conclusion hence follows by a simple rearrangement of the terms in \eqref{e4}.
\end{proof}

Now, we are able to estimate the Fourier transform.
\begin{prop}\label{p1}
	Let $I(\phi)$ and $H(\phi,\alpha)$ be as previously defined, and let $\tilde{\phi}=\frac{\pi}{2}-\frac{\phi}{2}$. Uniformly for every $\omega\in[-\varepsilon,\varepsilon]$, it holds
	\begin{equation*}
	\int_{I(\phi)}\int_{0}^{1}\left|\widehat{\mathds{1}}_{[\delta, \theta]H(\phi,\alpha)}\left(\rho\,\mathbf{u}(\tilde{\phi}+\omega)\right)\right|^2\de \delta\de \theta\asymp\begin{cases}
		\rho^{-3-\frac{1}{\alpha}}&\textnormal{if}\quad|\omega|\leq\rho^{\frac{1-\alpha}{\alpha}}\\
		\rho^{-3}\omega^{\frac{1}{\alpha-1}}&\textnormal{if}\quad\rho^{\frac{1-\alpha}{\alpha}}<|\omega|\leq\varepsilon
	\end{cases}.
	\end{equation*}
\end{prop}
\begin{proof}
	By symmetry, we can restrict ourselves to study the case of $\omega\in[0,\varepsilon]$. Indeed, by Lemma~\ref{t3}, we have that, uniformly for every $\omega\in[0,\varepsilon]$, it holds
	\begin{equation*}
	\begin{split}
	\int_{I(\phi)}\int_{0}^{1}\left|\widehat{\mathds{1}}_{[\delta,\theta] H(\phi,\alpha)}\!\left(\rho\,\mathbf{u}(\tilde{\phi}+\omega)\right)\right|^2\de \delta\de \theta&\asymp\rho^{-2}\int_{-\phi/2}^{\phi/2}\gamma_{[\theta]H(\phi,\alpha)}^2\!\left(\tilde{\phi}+\omega,\rho^{-1}\right)\de \theta\\
	&=\rho^{-2}\int_{\omega-\phi}^{\omega}\gamma_{H(\phi,\alpha)}^2\!\left(\frac{\pi}{2}+\theta,\rho^{-1}\right)\de \theta\\
	&\asymp\rho^{-2}\int_{-\phi/2}^{\omega}\gamma_{H(\phi,\alpha)}^2\!\left(\frac{\pi}{2}+\theta,\rho^{-1}\right)\de \theta,
	\end{split}
	\end{equation*}
	where the last approximation follows from the symmetries of $H(\phi,\alpha)$. By Proposition~\ref{l8} and by Proposition~\ref{l9}, we get that
	\begin{equation*}
	\gamma_{H(\phi,\alpha)}\left(\frac{\pi}{2}+\theta,\rho^{-1}\right)\asymp\begin{cases}
		-\rho^{-1}\theta^{-1}&\text{if}\quad-\frac{\phi}{2}\leq\theta<-\rho^{\frac{1-\alpha}{\alpha}}\\ \rho^{-1/\alpha}&\text{if}\quad|\theta|\leq\rho^{\frac{1-\alpha}{\alpha}}\\ \rho^{-1/2}\theta^{\frac{2-\alpha}{2(\alpha-1)}}&\text{if}\quad\rho^{\frac{1-\alpha}{\alpha}}<\theta\leq\varepsilon
	\end{cases}.
	\end{equation*}
	Therefore, uniformly for every $\omega\in\left[0,\rho^{\frac{1-\alpha}{\alpha}}\right]$, we have
	\begin{equation*}
	\begin{split}
		\int_{-\frac{\phi}{2}}^{\omega}\gamma_{H(\phi,\alpha)}^2\left(\frac{\pi}{2}+\theta,\rho^{-1}\right)\de \theta&\asymp\int_{-\frac{\phi}{2}}^{-\rho^\frac{1-\alpha}{\alpha}}\rho^{-2}|\theta|^{-2}\de \theta+\int_{-\rho^{\frac{1-\alpha}{\alpha}}}^\omega\rho^{-2/\alpha}\de \theta\\
		&=\rho^{-2}\left(\rho^\frac{\alpha-1}{\alpha}-2\phi^{-1}\right)+\rho^{-2/\alpha}\left(\omega+\rho^{\frac{1-\alpha}{\alpha}}\right)\asymp \rho^{-\frac{\alpha+1}{\alpha}}.
	\end{split}
	\end{equation*}
	On the other hand, in the case of $\omega\in\left(\rho^{\frac{1-\alpha}{\alpha}},\varepsilon\right]$, we must take into account the additional term
	\begin{equation*}
	\begin{split}
	\int_{\rho^{\frac{1-\alpha}{\alpha}}}^{\omega}\gamma_{H(\phi,\alpha)}^2\left(\frac{\pi}{2}+\theta,\rho^{-1}\right)\de \theta&\asymp \int_{\rho^{\frac{1-\alpha}{\alpha}}}^{\omega}\rho^{-1}\theta^{\frac{2-\alpha}{\alpha-1}}\de \theta\\
	&=\rho^{-1}(\alpha-1)\left(\omega^{\frac{1}{\alpha-1}}-\rho^{-1/\alpha}\right),
	\end{split}
	\end{equation*}
	and the initial claim easily follows.
\end{proof}

	We have gathered the necessary estimate to prove the main result of this section, namely that, for the affine quadratic discrepancy, all the intermediate polynomial orders between $N^{2/5}$ and $N^{1/2}$ are achievable.
	\begin{thm}\label{InterTeo}
		Let $I(\phi)$ and $\alpha$ be as in \eqref{NotazInter}. There exists a convex body $C(\phi,\alpha)$ with piecewise-$\mathcal{C}^\infty$ boundary such that it holds
		\begin{equation*}
		\inf_{\# \mathcal{P}=N}\mathcal{D}_2(\mathcal{P},\, C(\phi,\alpha),\,I(\phi))\asymp N^{\frac{2\alpha}{1+4\alpha}}.
		\end{equation*} 
	\end{thm}
\begin{proof}
	Let $H(\phi,\alpha)$ be as previously defined, and consider
 \begin{equation*}
     C(\phi,\alpha)=\left[\frac{\phi}{2}-\frac{\pi}{2}\right]\!H(\phi,\alpha).
 \end{equation*}
 In particular, notice that $C(\phi,\alpha)$ is symmetric with respect to the $x$-axis. Further, by Proposition~\ref{p1}, we have that, uniformly for every $\omega\in(-\varepsilon,\varepsilon)$, it holds
	\begin{equation}\label{e5}
	\int_{I(\phi)}\int_{0}^{1}\left|\widehat{\mathds{1}}_{[\delta,\theta] C(\phi,\alpha)}\left(\rho\,\mathbf{u}(\omega)\right)\right|^2\de \delta\de \theta\asymp\begin{cases}
		\rho^{-3-\frac{1}{\alpha}}&\text{if}\quad|\omega|\leq\rho^{\frac{1-\alpha}{\alpha}}\\
		\rho^{-3}\omega^{\frac{1}{\alpha-1}}&\text{if}\quad\rho^{\frac{1-\alpha}{\alpha}}<|\omega|\leq\varepsilon
	\end{cases},
	\end{equation}
    and, by symmetry, analogous estimates hold in the case of $\omega\in(\pi-\varepsilon,\pi+\varepsilon)$. On the other hand, since by construction $\partial C(\phi,\alpha)$ is $\mathcal{C}^\infty$ everywhere except at the origin and at its symmetric counterpart, by Lemma~\ref{t3} and Corollary~\ref{c2}, we have that, uniformly for every
 \begin{equation*}
     \omega\in[\varepsilon,\pi-\varepsilon]\cup[\pi+\varepsilon,2\pi-\varepsilon],
 \end{equation*}
 it holds
	\begin{equation}\label{e8}
		\int_{I(\phi)}\int_{0}^{1}\left|\widehat{\mathds{1}}_{[\delta] C(\phi,\alpha)}\left(\rho\,\mathbf{u}(\omega-\theta)\right)\right|^2\de \delta\de \theta\asymp\rho^{-3}.
	\end{equation}
	In particular, notice that the hypotheses of Theorem~\ref{t4} are satisfied, and we may apply it in the case of $h=1/\alpha$. Consequently, we get the lower bound
	\begin{equation*}
	\inf_{\# \mathcal{P}=N}\mathcal{D}_2(\mathcal{P},\, C(\phi,\alpha),\,I(\phi))\succcurlyeq
	N^{\frac{2\alpha}{1+4\alpha}}.
	\end{equation*}\par
 
    Now, we turn our attention to the upper bound and show it by constructing suitable samplings. First, let us do it for a number of $N$ points such that
    \begin{equation*}
        N=\lfloor n^{\frac{1+2\alpha}{1+4\alpha}}\rfloor\,\lfloor n^{\frac{2\alpha}{1+4\alpha}}\rfloor\quad\text{for some}\quad n\in\mathbb{N}.
    \end{equation*}
    Hence, set
    \begin{equation*}
        G=\lfloor n^{\frac{1+2\alpha}{1+4\alpha}}\rfloor,\quad L=\lfloor n^{\frac{2\alpha}{1+4\alpha}}\rfloor,\quad J_G=[0,G-1]\cap\mathbb{N},\quad\text{and}\quad J_L=[0,L-1]\cap\mathbb{N}. 
    \end{equation*}
    Consider the set of points $\mathcal{P}_N\subset\mathbb{T}^2$ defined by
	\begin{equation*}
	\mathcal{P}_N=\left\{\mathbf{p}_j\right\}_{j=1}^N=\{\mathbf{p}_{\ell,g}\}_{\ell\in J_L,\,g\in J_G}\quad\text{with}\quad \mathbf{p}_{\ell,g}=\left(\frac{\ell}{L},\,\frac{g}{G}\right),
	\end{equation*}
	where the coordinates of $\mathbf{p}_{\ell,g}$ are to be intended modulo $1$. Again, by Parseval's identity, we get
	\begin{equation*}
	\int_{\mathbb{T}^2}\left|\mathcal{D}(\mathcal{P}_N, [\boldsymbol{\tau},\delta, \theta] C(\phi,\alpha))\right|^2\de \boldsymbol{\tau}=\sum_{\mathbf{m}\neq(0,0)}\left|\widehat{\mathds{1}}_{[\delta, \theta] C(\phi,\alpha)}(\mathbf{m})\right|^2\left|\sum_{g\in J_G}\sum_{\ell\in J_L}e^{2\pi i \mathbf{m}\cdot \mathbf{p}_{\ell,g}}\right|^2,
	\end{equation*}
	and we observe that
	\begin{equation*}
	\sum_{g\in J_G}\sum_{\ell\in J_L}e^{2\pi i\left(m_1\frac{\ell}{L}+m_2\frac{g}{G}\right)}=
	\begin{cases}
	GL &\text{if}\quad m_1\in L\mathbb{Z}\quad\text{and}\quad m_2\in G\mathbb{Z}\\
		0 &\text{else}
	\end{cases}.
	\end{equation*}
	Hence, we can consider
	\begin{equation*}
	\mathbf{m}=(Ln_1,Gn_2)\quad \text{with}\quad \mathbf{n}\in\mathbb{Z}^2,
	\end{equation*}
	and split the set
 \begin{equation*}
     \mathcal{R}=(L\mathbb{Z}\times G\mathbb{Z})\setminus\{\mathbf{0}\}
 \end{equation*}
 into the regions
	\begin{equation*}
	\begin{split}
	V_1&=\left\{\mathbf{m}\in\mathcal{R}\,\colon\,|m_2|^\alpha\leq|m_1|\right\},\\
	V_2&=\left\{\mathbf{m}\in\mathcal{R}\,\colon\,|m_2|\leq|m_1|<|m_2|^\alpha\right\},\\
	V_3&=\left\{\mathbf{m}\in\mathcal{R}\,\colon\,|m_1|<|m_2|\right\}.
	\end{split}
	\end{equation*}
	Then, we write
	\begin{equation}\label{e7}
	\begin{split}
		&\mathcal{D}_2(\mathcal{P}_N,\, C(\phi,\alpha),\,I(\phi))=\\
    &=\int_{I(\phi)}\int_{0}^{1}G^2L^2\sum_{\mathbf{m}\in\mathcal{R}}\left|\widehat{\mathds{1}}_{[\delta,\theta] C_\alpha^\theta}(\mathbf{m})\right|^2\de\delta\de \theta\\
		&=G^2L^2\left(\sum_{\mathbf{m}\in V_1}+\sum_{\mathbf{m}\in V_2}+\sum_{\mathbf{m}\in V_3}\right)\int_{I(\phi)}\int_{0}^{1}\left|\widehat{\mathds{1}}_{[\delta, \theta] C(\phi,\alpha)}(\mathbf{m})\right|^2\de\delta\de \theta.
	\end{split}
	\end{equation}
        We exploit \eqref{e5} and \eqref{e8} in order to study the three sums in the latter equation. In this case, we must consider
        \begin{equation*}
            \rho=|\mathbf{m}|\quad\text{and}\quad\tan\omega=\frac{m_2}{m_1}.
        \end{equation*}
        We notice that for $\omega\in[-1,1]$ it holds $\tan\omega\approx\omega$, and consequently, with a bit of rearrangement, we can rewrite the estimates in \eqref{e5} and \eqref{e8} as
        \begin{equation}\label{e5-2}
	\int_{I(\phi)}\int_{0}^{1}\left|\widehat{\mathds{1}}_{[\delta,\theta] C(\phi,\alpha)}\left(\mathbf{m}\right)\right|^2\de \delta\de \theta\asymp\begin{cases}
		|m_1|^{-3-\frac{1}{\alpha}}&\text{if }|m_2|^\alpha\leq|m_1|\\
		|m_1|^{\frac{2-3\alpha}{\alpha-1}}|m_2|^{\frac{1}{\alpha-1}}&\text{if }|m_2|\leq|m_1|<|m_2|^\alpha\\
        |m_2|^{-3}
        &\text{if }|m_1|<|m_2|
	\end{cases}.
	\end{equation}
	By the latter, for the first sum in the last term of \eqref{e7}, we get
	\begin{align*}
		&G^2L^2\sum_{\mathbf{m}\in V_1}\int_{I(\phi)}\int_{0}^{1}\left|\widehat{\mathds{1}}_{[\delta, \theta] C(\phi,\alpha)}(\mathbf{m})\right|^2\de\delta\de \theta\\
        &\preccurlyeq G^2L^2\sum_{\mathbf{m}\in V_1}|m_1|^{-3-\frac{1}{\alpha}}\\
		&\preccurlyeq G^2 L^2\sum_{n_1=1}^{+\infty}\,\sum_{n_2=0}^{n_1^{1/\alpha}L^{1/\alpha}G^{-1}}(Ln_1)^{-3-\frac{1}{\alpha}}\\
        &\preccurlyeq G^2 L^{-\frac{1+\alpha}{\alpha}}\sum_{n_1=1}^{+\infty}n_1^{-3-\frac{1}{\alpha}}\left(1+n_1^{1/\alpha}L^{1/\alpha}G^{-1}\right)\\
		&\preccurlyeq\left(G^2L^{-\frac{1+\alpha}{\alpha}}+GL^{-1}\right)\preccurlyeq N^{\frac{2\alpha}{1+4\alpha}}.
	\end{align*}
	
        For the second sum in the last term of \eqref{e7}, by applying \eqref{e5-2}, we get
	\begin{align*}
		&G^2L^2\sum_{\mathbf{m}\in V_2}\int_{I(\phi)}\int_{0}^{1}\left|\widehat{\mathds{1}}_{[\delta, \theta] C(\phi,\alpha)}(\mathbf{m})\right|^2\de\delta\de \theta\\
        &\preccurlyeq G^2L^2\sum_{\mathbf{m}\in V_2}|m_1|^{\frac{2-3\alpha}{\alpha-1}}|m_2|^{\frac{1}{\alpha-1}}\\
		&\preccurlyeq G^2 L^2\sum_{n_2=1}^{+\infty}\,\sum_{n_1=n_2GL^{-1}}^{n_2^\alpha G^\alpha L^{-1}}(Ln_1)^{\frac{2-3\alpha}{\alpha-1}}(Gn_2)^{\frac{1}{\alpha-1}}\\
		&\preccurlyeq G^{\frac{2\alpha-1}{\alpha-1}} L^{\frac{-\alpha}{\alpha-1}}\sum_{n_2=1}^{+\infty}n_2^{\frac{1}{\alpha-1}}\sum_{n_1=n_2GL^{-1}}^{+\infty}n_1^{\frac{2-3\alpha}{\alpha-1}}\\
		&\preccurlyeq G^{\frac{2\alpha-1}{\alpha-1}} L^{\frac{-\alpha}{\alpha-1}}\sum_{n_2=1}^{+\infty}n_2^{\frac{1}{\alpha-1}}\left(n_2GL^{-1}\right)^{\frac{1-2\alpha}{\alpha-1}}\\
        &\preccurlyeq L\sum_{n_2=1}^{+\infty}n_2^{-2}\preccurlyeq L\preccurlyeq N^{\frac{2\alpha}{1+4\alpha}}.
	\end{align*}
	Finally, for the last sum in the last term in \eqref{e7}, again by applying \eqref{e5-2}, we get
	\begin{align*}
		&G^2L^2\sum_{\mathbf{m}\in V_3}\int_{I(\phi)}\int_{0}^{1}\left|\widehat{\mathds{1}}_{[\delta, \theta] C(\phi,\alpha)}(\mathbf{m})\right|^2\de\delta\de \theta\\
        &\preccurlyeq G^2L^2\sum_{\mathbf{m}\in V_3}|m_2|^{-3}\\
		&\preccurlyeq G^2 L^2\sum_{n_2=1}^{+\infty}\,\sum_{n_1=0}^{n_2GL^{-1}}(Gn_2)^{-3}\\
        &\preccurlyeq G^{-1}L^2\sum_{n_2=1}^{+\infty}n_2^{-2}GL^{-1}\preccurlyeq L\preccurlyeq N^{\frac{2\alpha}{1+4\alpha}}.
	\end{align*}
	Hence, we can conclude that the upper bound holds for all $N$ of the form $N=\lfloor n^{\frac{1+2\alpha}{1+4\alpha}}\rfloor\,\lfloor n^{\frac{2\alpha}{1+4\alpha}}\rfloor$.\par
	Last, in order to prove the initial claim holds for every $N\in\mathbb{N}$, it is enough to repeat the argument at the end of Theorem~\ref{t6} with adjusted exponents.
\end{proof}

\section*{Acknowledgements} I am grateful to my advisors, Luca Brandolini, Leonardo Colzani, Giacomo Gigante, and Giancarlo Travaglini, for their support and all the valuable discussions.

\newpage

\bibliography{main.bib}

\begin{thebibliography}{CDMM90}

\bibitem[BDP20]{MR4391422}
Dmitriy Bilyk, Josef Dick, and Friedrich Pillichshammer, editors.
\newblock {\em Discrepancy theory}.
\newblock De Gruyter, Berlin, 2020.

\bibitem[Bec87]{MR0906524}
J\'{o}zsef Beck.
\newblock Irregularities of distribution. {I}.
\newblock {\em Acta Math.}, 159(1-2):1--49, 1987.

\bibitem[BG24]{beretti2024fouriertransformbvfunctions}
Thomas Beretti and Luca Gennaioli.
\newblock Fourier transform of {BV} functions, isoperimetry, and discrepancy theory (preprint), 2024.

\bibitem[BGG21]{MR4238281}
Luca Brandolini, Bianca Gariboldi, and Giacomo Gigante.
\newblock On a sharp lemma of {C}assels and {M}ontgomery on manifolds.
\newblock {\em Math. Ann.}, 379(3-4):1807--1834, 2021.

\bibitem[BHI03]{MR2006553}
L.~Brandolini, S.~Hofmann, and A.~Iosevich.
\newblock Sharp rate of average decay of the {F}ourier transform of a bounded set.
\newblock {\em Geom. Funct. Anal.}, 13(4):671--680, 2003.

\bibitem[Bil11]{MR2817765}
Dmitriy Bilyk.
\newblock On {R}oth's orthogonal function method in discrepancy theory.
\newblock {\em Unif. Distrib. Theory}, 6(1):143--184, 2011.

\bibitem[BM23]{MR4585469}
Dmitriy Bilyk and Michelle Mastrianni.
\newblock Lower bounds for the directional discrepancy with respect to an interval of rotations.
\newblock {\em J. Fourier Anal. Appl.}, 29(3), 2023.

\bibitem[BMPS11]{MR2861537}
Dmitriy Bilyk, Xiaomin Ma, Jill Pipher, and Craig Spencer.
\newblock Directional discrepancy in two dimensions.
\newblock {\em Bull. Lond. Math. Soc.}, 43(6):1151--1166, 2011.

\bibitem[BMPS16]{MR3453360}
Dmitriy Bilyk, Xiaomin Ma, Jill Pipher, and Craig Spencer.
\newblock Diophantine approximations and directional discrepancy of rotated lattices.
\newblock {\em Trans. Amer. Math. Soc.}, 368(6):3871--3897, 2016.

\bibitem[BNW88]{BNW88}
Joaquim Bruna, Alexander Nagel, and Stephen Wainger.
\newblock Convex hypersurfaces and {F}ourier transforms.
\newblock {\em Ann. of Math. (2)}, 127(2):333--365, 1988.

\bibitem[BT22]{MR4358540}
Luca Brandolini and Giancarlo Travaglini.
\newblock Irregularities of distribution and geometry of planar convex sets.
\newblock {\em Adv. Math.}, 396, 2022.

\bibitem[Cas56]{MR0087709}
J.~W.~S. Cassels.
\newblock On the sums of powers of complex numbers.
\newblock {\em Acta Math. Acad. Sci. Hungar.}, 7:283--289, 1956.

\bibitem[CDMM90]{CDMM90}
Michael Cowling, Shaun Disney, Giancarlo Mauceri, and Detlef M\"uller.
\newblock Damping oscillatory integrals.
\newblock {\em Invent. Math.}, 101(2):237--260, 1990.

\bibitem[Cha00]{MR1779341}
Bernard Chazelle.
\newblock {\em The discrepancy method}.
\newblock Cambridge University Press, Cambridge, 2000.

\bibitem[CST14]{MR3307667}
William Chen, Anand Srivastav, and Giancarlo Travaglini, editors.
\newblock {\em A panorama of discrepancy theory}.
\newblock Springer, Cham, 2014.

\bibitem[CV17]{MR3731299}
William W.~L. Chen and Robert~C. Vaughan.
\newblock In memoriam {K}laus {F}riedrich {R}oth 1925--2015.
\newblock {\em Mathematika}, 63(3):711--712, 2017.

\bibitem[Dav56]{MR0082531}
H.~Davenport.
\newblock Note on irregularities of distribution.
\newblock {\em Mathematika}, 3:131--135, 1956.

\bibitem[Dic14]{MR3330354}
Josef Dick.
\newblock Applications of geometric discrepancy in numerical analysis and statistics.
\newblock In {\em Applied algebra and number theory}, pages 39--57. Cambridge Univ. Press, Cambridge, 2014.

\bibitem[DL93]{MR1261635}
Ronald~A. DeVore and George~G. Lorentz.
\newblock {\em Constructive approximation}.
\newblock Springer-Verlag, Berlin, 1993.

\bibitem[Drm96]{MR1401710}
Michael Drmota.
\newblock Irregularities of distributions with respect to polytopes.
\newblock {\em Mathematika}, 43(1):108--119, 1996.

\bibitem[DT97]{MR1470456}
Michael Drmota and Robert~F. Tichy.
\newblock {\em Sequences, discrepancies and applications}.
\newblock Springer-Verlag, Berlin, 1997.

\bibitem[GG22]{MR4480211}
Bianca Gariboldi and Giacomo Gigante.
\newblock Almost positive kernels on compact {R}iemannian manifolds.
\newblock {\em Math. Z.}, 302(2):783--801, 2022.

\bibitem[Her62]{Herz62}
C.~S. Herz.
\newblock Fourier transforms related to convex sets.
\newblock {\em Ann. of Math. (2)}, 75:81--92, 1962.

\bibitem[Hla50]{Hla50}
Edmund Hlawka.
\newblock {\"U}ber {I}ntegrale auf konvexen {K}\"orpern. {I}.
\newblock {\em Monatsh. Math.}, 54:1--36, 1950.

\bibitem[Ken48]{MR0024929}
David~G. Kendall.
\newblock On the number of lattice points inside a random oval.
\newblock {\em Quart. J. Math. Oxford Ser.}, 19:1--26, 1948.

\bibitem[Mat10]{MR2683232}
Ji\v{r}\'{\i} Matou\v{s}ek.
\newblock {\em Geometric discrepancy}.
\newblock Springer-Verlag, Berlin, 2010.

\bibitem[Mon94]{MR1297543}
Hugh~L. Montgomery.
\newblock {\em Ten lectures on the interface between analytic number theory and harmonic analysis}.
\newblock American Mathematical Society, Providence, 1994.

\bibitem[Pod91]{MR1166380}
A.~N. Podkorytov.
\newblock On the asymptotics of the {F}ourier transform on a convex curve.
\newblock {\em Vestnik Leningrad. Univ. Mat. Mekh. Astronom.}, 1991.

\bibitem[Ran69a]{Ran69-2}
Burton Randol.
\newblock On the asymptotic behavior of the {F}ourier transform of the indicator function of a convex set.
\newblock {\em Trans. Amer. Math. Soc.}, 139:279--285, 1969.

\bibitem[Ran69b]{Rand69-1}
Burton Randol.
\newblock On the {F}ourier transform of the indicator function of a planar set.
\newblock {\em Trans. Amer. Math. Soc.}, 139:271--278, 1969.

\bibitem[Rot54]{MR66435}
K.~F. Roth.
\newblock On irregularities of distribution.
\newblock {\em Mathematika}, 1:73--79, 1954.

\bibitem[Sch69]{MR0245532}
Wolfgang~M. Schmidt.
\newblock Irregularities of distribution. {IV}.
\newblock {\em Invent. Math.}, 7:55--82, 1969.

\bibitem[Sie35]{MR1555407}
Carl~Ludwig Siegel.
\newblock {\"U}ber {G}itterpunkte in {C}onvexen {K}\"orpern und ein {D}amit {Z}usammenh\"angendes {E}xtremalproblem.
\newblock {\em Acta Math.}, 65(1):307--323, 1935.

\bibitem[Tra14]{MR3307692}
Giancarlo Travaglini.
\newblock {\em Number theory, {F}ourier analysis and geometric discrepancy}.
\newblock Cambridge University Press, Cambridge, 2014.

\bibitem[TT16]{MR3500235}
G.~Travaglini and M.~R. Tupputi.
\newblock A characterization theorem for the {$L^2$}-discrepancy of integer points in dilated polygons.
\newblock {\em J. Fourier Anal. Appl.}, 22(3):675--693, 2016.

\bibitem[{van}35]{zbMATH03019347}
J.~G. {van der Corput}.
\newblock Verteilungsfunktionen. {I}.
\newblock {\em Proc. Akad. Wet. Amsterdam}, 38:813--821, 1935.

\bibitem[{van}45]{MR0015143}
T.~{van Aardenne-Ehrenfest}.
\newblock Proof of the impossibility of a just distribution of an infinite sequence of points over an interval.
\newblock {\em Nederl. Akad. Wetensch., Proc.}, 48:71--76, 1945.

\bibitem[{van}49]{MR32717}
T.~{van Aardenne-Ehrenfest}.
\newblock On the impossibility of a just distribution.
\newblock {\em Nederl. Akad. Wetensch., Proc.}, 52:734--739, 1949.

\end{thebibliography}
\bibliographystyle{alpha}

\end{document}